\documentclass[12pt]{article}
\usepackage{bbm, fancyhdr}
\usepackage{amssymb}
\usepackage{mathrsfs}
\usepackage{mathtools}
\usepackage{amsmath}
\usepackage{amsthm}
\usepackage[utf8]{inputenc}
\usepackage[english]{babel}
\usepackage{comment}
\usepackage{listings}
\usepackage{enumitem}
\usepackage{scalerel}
\usepackage{todonotes}
\usepackage{stackengine,wasysym}
\usepackage[backref=page]{hyperref}
\renewcommand*{\backref}[1]{}
\renewcommand*{\backrefalt}[4]{%
    \ifcase #1 (Not cited.)%
    \or        (Cited on page~#2.)%
    \else      (Cited on pages~#2.)%
    \fi}
\usepackage[utf8]{inputenc}
\usepackage{float}

\hypersetup{colorlinks,
            breaklinks,
            urlcolor=magenta,
            linkcolor=magenta}

\usepackage{amsthm}
\usepackage[capitalise,nameinlink]{cleveref}

\newcommand\N{\mathbb{N}}

\newcommand\R{\mathbb{R}}
\newcommand\C{\mathbb{C}}

\newcommand\6{\partial}

\newcommand\End{\textrm{End}}

\newcommand\Spec{\mathrm{Spec}}

\newcommand\Id{\textrm{Id}}
\newcommand{\bas}{\textrm{bas}}
\newcommand{\Sym}{\textrm{Sym}}

\renewcommand\Im{\textrm{Im}}

\renewcommand\Re{\textrm{Re}}

\newcommand\pr{\textrm{pr}}
\newcommand\tr{\textrm{tr}}
\newcommand\Ric{\textrm{Ric}}
\newcommand\codim{\textrm{codim}}

\newcommand\Fr{\textrm{Fr}}
\newcommand\Cal{\textrm{Cal}}

\newcommand\Symm{\textrm{Symm}}
\newcommand\vol{\textrm{vol}}
\newcommand\deRham{\textrm{dR}}
\newcommand\Alt{\textrm{Alt}}
\newcommand\Hess{\textrm{Hess}}
\newcommand\Tr{\textrm{Tr}}

\newcounter{globcounter}[section]

\newtheorem{theorem}[globcounter]{Theorem}
\newtheorem{proposition}[globcounter]{Proposition}
\newtheorem{proposition_definition}[globcounter]{Proposition/Definition}
\newtheorem{corollary}[globcounter]{Corollary}
\newtheorem{lemma}[globcounter]{Lemma}
\theoremstyle{definition}
\newtheorem{remark}[globcounter]{Remark}
\theoremstyle{definition}

\theoremstyle{definition}
\newtheorem{definition}[globcounter]{Definition}

\crefname{theorem}{theorem}{theorems}
\crefname{theorem}{Theorem}{Theorems}
\crefname{lemma}{lemma}{lemmas}
\crefname{lemma}{Lemma}{Lemmas}
\crefname{proposition}{proposition}{propositions}
\crefname{proposition}{Proposition}{Propositions}
\crefname{definition}{definition}{definitions}
\crefname{definition}{Definition}{Definitions}

\newcounter{proofstepcounter}[globcounter]
\begin{document}
	\begin{center}
	{\Large\bf  An Aubin-Yau theorem for transversally K\"ahler foliations }\\[5mm]
	{\large Vlad Marchidanu\footnote{Partly supported by the PNRR-III-C9-2023-I8 grant CF 149/31.07.2023 {\em Conformal Aspects of Geometry and Dynamics}.\\[1mm]
		
		
		}}\\[4mm]
	\end{center}
	
	{\small
		\hspace{0.15\linewidth}
		\begin{minipage}[t]{0.7\linewidth}
			{\bf Abstract.} Transversally K\"ahler foliations are a generalisation of K\"ahler manifolds, appearing naturally in the complex non-K\"ahler setting. We give a self-contained proof of how the classical methods used in the proof of the Aubin-Yau theorem adapt to the transversally K\"ahler case under the homological orientability condition. We apply this result to obtain a new, simpler proof of the already known Vaisman Aubin-Yau theorem. 
   \\
		\end{minipage}
	}

\tableofcontents
 
\section{Introduction}

Given a complex manifold, one is naturally inclined to ask for the existence of special Hermitian metrics. From a Riemannian perspective, some metrics of distinguished interest are the Einstein ones, for which the Ricci curvature is in as good a relationship with the metric as can be. Finding Hermitian Einstein metrics has proven to be a difficult task, as is shown in the extended survey \cite{besse}. In the K\"ahler world, it has become known after a fair amount of work that one has much more control than in the generic complex setting. One reason for this, the authors of \cite{besse} observe, is the relative autonomy of the Ricci tensor with respect to the metric, given a fixed complex structure.

It was Calabi in \cite{calabi_1955_on_Kaehler_with_vanishing_canonical_class} who conjectured that on a compact K\"ahler manifold, up to a constant multiple any real $(1,1)$ representative of the first Chern class can be realised as the Ricci form of some K\"ahler metric. He proved the uniqueness of this metric in the cohomology class of an initial K\"ahler metric and proposed using the Schauder continuity method for deriving existence, towards which he showed the openness of a certain solution set. It took more than 20 years for this method to successfully achieve its end, with Yau's proof in \cite{yau_1978_on_ricci_curvature_and_monge_ampere} of the necessary $\mathcal{C}^2$ and $\mathcal{C}^3$ estimates. Yau's method of proof has been simplified in the meantime by the efforts of Kazdan, Aubin, and Bourguignon, exposed e.g. in \cite{blocki}. Proofs which depend solely on local considerations (and are thus more analytical) can be found also in \cite{blocki} or in \cite{szekelyhidi_intro_extremal}.

A closely related result is that of Aubin from
\cite{aubin_1976_equations_monge_ampere}, which tackles the situation when on a compact complex manifold the first Chern class is negative definite; in this case, one finds a unique K\"ahler-Einstein metric with Einstein constant $-1$. This result was reproved by Yau in the same paper \cite{yau_1978_on_ricci_curvature_and_monge_ampere}, using similar estimation techniques.

A generalisation of K\"ahler manifolds are transversally K\"ahler foliations. As the name suggests, they are foliations for which any smooth local leaf space is a K\"ahler manifold. They appear in different areas of interest related to complex geometry. For example, any Vaisman manifold is endowed with a canonical transversally K\"ahler foliation (see \cref{sec::application_vaisman}). They also appear at the intersection of contact and complex geometry: the distribution generated by the Reeb vector field on a Sasakian manifold integrates to a transversally K\"ahler foliation.

There is a natural way to formulate analogues of the Calabi-Yau and Aubin-Yau problems for transversally K\"ahler foliations, using the first Chern class of the transverse fibre bundle of the foliation, seen as an element in basic cohomology (see \cref{sec::transversally_kaheler_foliatons} and the statement of \cref{thm::transversal_aubin_yau}). 
In this regard, the only result which is present in the literature concerns a version of the Calabi-Yau theorem for transversally K\"ahler foliations. All authors mentioning this result refer to the paper of El Kacimi-Alaoui, \cite{alaoui_1990_operateurs_transversalement_elliptiques}. However, in \cite{alaoui_1990_operateurs_transversalement_elliptiques} the result is only claimed to follow as in the classical case. On the other hand, an analogue of the $\6\bar\6$-lemma in the transversally K\"ahler case is arrived at in this paper. This is essential for a proof of transverse versions of the Calabi-Yau and Aubin-Yau theorems, as explained in \cref{sec::transversally_kaheler_foliatons} and \cref{sec::aubin_yau_transv_kaehler}.

The purpose of this note is to give a detailed and self-contained proof of an analogue of the Aubin-Yau theorem for transversally K\"ahler foliations. We are careful to highlight at each step why, how, and under what conditions can the classical techniques adapt to the transversally K\"ahler case. To the best of the author's knowledge, such a result is missing from the literature. This lack has also been signalled in the beginning of \cite[Section~6]{istrati_2025_Vaisman_manifolds_vanishing_first_chern_class}.

As an application, we give a new, simpler proof of the Aubin-Yau theorem for Vaisman manifolds, first done in \cite{istrati_2025_Vaisman_manifolds_vanishing_first_chern_class} by using a Weitzenb\"ock-type formula. Note that the statement of \cite[Theorem~6.3]{istrati_2025_Vaisman_manifolds_vanishing_first_chern_class} doesn't coincide with ours, but it is equivalent; see the discussion before \cref{thm::Vaisman_Aubin_Yau}.

The paper is organised as follows. 
In \cref{sec::review_foliations_Molino_theory} we give a brief review of foliation theory. We focus on the viewpoint of Molino theory, recalling some useful facts about transverse $G$-structures. Using these, we recall the definition of the basic Laplace-Beltrami operator, which is a self-adjoint operator on basic functions. We then compare it to the transverse Laplace-Beltrami operator. 
In \cref{sec::transversally_kaheler_foliatons} we define transversally K\"ahler foliations and draw a parallel between the classical K\"ahler case and the transversally K\"ahler one. 
In \cref{sec::elliptic_operators_on_transv_kaehler}, we focus on the analysis required for our main result. We define H\"older spaces of basic functions and prove an essential property they have. Afterwards, we use the facts in \cref{sec::review_foliations_Molino_theory} to prove a result on the existence and uniqueness of solutions of an operator which involves the transverse Laplace-Beltrami operator, giving a simple and relatively self-contained proof. We then focus on the fully nonlinear second order differential equation in basic functions which encodes the transversally K\"ahler Aubin-Yau problem. For this equation, we prove the existence of solutions - to the best of our knowledge, for the first time in the literature. 
In \cref{sec::aubin_yau_transv_kaehler} we state and finish the proof of our main result, an Aubin-Yau theorem for transversally K\"ahler foliations. We show how to reduce it to the proof of the existence and uniqueness of solutions for the aforementioned equation and then finish the proof by proving uniqueness. 
In \cref{sec::application_vaisman} we define LCK and Vaisman manifolds, recall the definition of the weight bundle of LCK manifolds, then we re-state and re-prove the Aubin-Yau theorem for Vaisman manifolds, giving an alternative, simpler proof than the one which already exists.

\vspace{2em}

\textit{Acknowledgements.}
The author is grateful to Liviu Ornea for his constant support and patience throughout this endeavour. The author is also thankful to Misha Verbitsky and to Nicolina Istrati for the very helpful discussions on this subject.

\section{Short review of foliations and Molino theory}
\label{sec::review_foliations_Molino_theory}
Throughout this section $M$ will denote a smooth connected manifold of real dimension $n$.
\begin{definition}\label{def::foliation}
    For $0 \leq p \leq n$ and $q:= n-p$, a (smooth) \textit{foliation} $\mathcal{F}$ of dimension $p$ on $M$ is a maximal smooth atlas with the property that all transition functions project to diffeomorphisms between open subsets in $\R^q$ locally, i.e. around each point in their domain. A chart compatible with the maximal atlas determining $\mathcal{F}$ is called a \textit{foliated} chart.
\end{definition}

\begin{remark}\label{rmk::foliation_equiv_defs_and_leaves}
    By the well-known Frobenius Theorem, this definition is equivalent to the presence of an integrable distribution. Moreover, for each point, the union of all immersed submanifolds which are integral for this distribution and pass through the point has a smooth manifold structure and is called the \textit{leaf} of the foliation passing through that point 
    (cf. {\nolinebreak \cite[Theorem~19.21~and~Lemma~19.22]{lee_intro_smooth}}
    for an infinitesimal approach to this result and 
    {\nolinebreak \cite[Section~1.3]{molino_riemannian_foliations_book}} for a global one based on path-connectedness).
\end{remark}

\begin{definition}\label{def::foliation_simple}
    A \textit{simple} foliation is a (smooth) foliation $(M,\mathcal{F})$ with the property that the quotient space $M/\mathcal{F}$, consisting of all the leaves of $\mathcal{F}$, admits a smooth manifold structure such that the quotient map ${\pi: M \rightarrow M/\mathcal{F}}$ is a smooth submersion with connected preimages.
\end{definition}

\begin{remark}\label{rmk::foliation_all_admit_simple_atlas}
     Any foliated manifold admits an atlas with foliated charts to which the restriction of the foliation is simple.
\end{remark}

\begin{definition}\label{def::basic_forms_basic_cohomology}
    Let $\alpha \in \Omega^r(M)$ for $r \in \{ 1, \ldots, q \}$. $\alpha$ is called a \textit{basic} form if for any vector field tangent to $\mathcal{F}$, $X \in \Gamma(T\mathcal{F})$, it holds that $\iota_X \alpha = 0$ and $\mathcal{L}_X \alpha = 0$. Basic forms are denoted $\Omega^*_\bas(M;\mathcal{F})$. A $0$-form i.e. a function $f \in \mathcal{C}^\infty(M)$ is called \textit{basic} if $X(f)=0$ for any $X \in \Gamma(T\mathcal{F})$.

    Since $d$ sends basic forms to basic forms, there is a natural \textit{basic de Rham cohomology}
    \[
        H^r_\bas(M;\mathcal{F}) := 
        \frac{
                \ker (d) \cap \Omega^r_\bas(M;\mathcal{F})
            }
            {
                d \left( \Omega^{r-1}_\bas(M;\mathcal{F}) \right)
            }
    \]
    and a natural map $H^r_\bas(M;\mathcal{F}) \rightarrow H^r_\deRham(M)$, which is in general neither injective nor surjective.
\end{definition}

The general difficulty in working with foliations is that their leaf spaces are not well behaved. Analogues of geometric structures on the leaf space are therefore used, which cleverly exploit the interplay between vector fields tangent to the foliation and those locally tangent to quotient spaces, in order to define global geometric objects which locally come from geometric objects on the leaf space of each simple chart.

\begin{definition}\label{def::transverse_frame_bundle}
    Let $(M, \mathcal{F})$ be a smooth foliation of codimension $q$. We call the quotient vector bundle ${\nu \mathcal{F}}:= TM/T\mathcal{F}$ the \textit{transverse fibre bundle} in the terminology of \cite{molino_riemannian_foliations_book} or the \textit{normal bundle} in that of \cite{tondeur_geometry_foliations}; we denote by $\overline X$ the induced projection of $X \in \Gamma(TM)$ to $\Gamma({\nu \mathcal{F}})$. 
    
    We denote by ${p: \Fr (\nu \mathcal{F}) \rightarrow M}$ the \textit{transverse frame bundle} i.e. the principal $GL(q, \R)$-bundle of frames in the transverse fibre bundle, seen as isomorphisms $z: \R^q \rightarrow {\nu \mathcal{F}}_{p(z)}$.
\end{definition}

\begin{definition}\label{def::lifted_foliation}
    Let $(M, \mathcal{F})$ be a smooth foliation of dimension $p$.
    Let $\Theta_T \in \Omega^1(\Fr (\nu \mathcal{F}); \R^q)$ be the $\R^q$-valued $1$-form on $\Fr (\nu \mathcal{F})$ defined by
    \[
        {
            \Theta_T(X_z):= z^{-1}\left(
                p_{*, z}(\overline{X}_z) 
            \right)
        }, \quad \forall
            {z \in \Fr (\nu \mathcal{F})}
       , \quad \forall
            {X \in T_z \Fr (\nu \mathcal{F})}.
    \]
    We consider the distribution on $\Fr (\nu \mathcal{F})$ defined as:
    \[
    D_T = \{ X \in \Gamma(\Fr (\nu \mathcal{F})): i_X \Theta_T = i_X d\Theta_T = 0 \}
    \]
    The foliation associated to $D_T$ is called the \textit{lifted foliation} and is denoted by $\mathcal{F}_T$.
\end{definition}

\begin{remark}\label{rmk::lifted_foliation_properties}
    \begin{enumerate}
        \item By \cite[Proposition~2.4]{molino_riemannian_foliations_book}, $D_T$ is an integrable distribution, hence, by \cref{rmk::foliation_equiv_defs_and_leaves}, $\mathcal{F}_T$ exists.
        \item By \cite[Lemma~2.1]{molino_riemannian_foliations_book}, the vectors tangent to $\mathcal{F}_T$ are exactly those which vanish when they are seen via push-forward as tangent to the (usual) frame bundle of the leaf space of any simple chart.
    \end{enumerate}
\end{remark}

We can now introduce global geometric objects which locally come from the leaf space of simple charts, following the definition of Conlon in \cite{conlon_transversally_parallelisable_foliations_codim_2} for the codimension $2$ case, introduced by Molino in full generality in \cite{molino_connexions_et_G_structures}.

\begin{definition}\label{def::transverse_G_structures}
    Let $(M, \mathcal{F})$ be a foliation of codimension $q$ and $G$ be a Lie subgroup of $GL(q, \R)$. Let $E$ be a principal $G$-subbundle of $\Fr (\nu F)$. We call $E$ a \textit{transverse $G$-structure} if all vectors tangent to the lifted foliation $\mathcal{F}_T$ are tangent to $E$ as well.
\end{definition}

\begin{remark}\label{rmk::riemannian_foliation_equivalent_defs}
 In particular, when $G = O(q)$ in \cref{def::transverse_G_structures}, we obtain a \textit{Riemannian foliation}, which can be characterised in another useful way.
Namely, there is an equivalence between Riemannian metrics on ${\nu \mathcal{F}}$ induced from transverse $O(q)$-structures, and \textit{holonomy invariant} (in the terminology of \cite[Chapter~5]{tondeur_geometry_foliations}) Riemannian metrics on ${\nu \mathcal{F}}$, i.e. nondegenerate tensors $g_T \in \Symm^2_\bas(M; \mathcal{F})$, or, more explicitly, Riemannian metrics $g_T$ on ${\nu \mathcal{F}}$ with ${\mathcal{L}_X g_T = 0}$, for any ${X \in \Gamma(T\mathcal{F})}$. More precisely, as in \cite[Proposition~3.2]{molino_riemannian_foliations_book}, a nonnegative bilinear symmetric form $g$ having kernel of constant dimension and satisfying the holonomy invariance property for all vector fields tangent to $\ker g$ has integrable kernel for which it is a transverse metric.
\end{remark}

\begin{definition}\label{def::transversally_parallelisable_foliation}
    A smooth foliation is called \textit{transversally parallelisable} if it admits a transversal $\{ 1 \}$-structure.
\end{definition}

Transversally parallelisable foliations are much better behaved than arbitrary ones. Whenever $G$ is a Lie subgroup of $GL(q, \R)$ and we have a transverse $G$-structure for $\mathcal{F}$, the lifted foliation restricted to the transverse $G$-structure is transversally parallelisable. In particular, when $\mathcal{F}$ is Riemannian (which is our case of interest) we have:

\begin{theorem}{\cite[Section~3.3]{molino_riemannian_foliations_book}}
    \label{thm::lifted_foliation_transversally_parallelisable}
    Let $(M, \mathcal{F})$ be a Riemannian foliation with transverse Levi-Civita connection $\omega$ on a transverse $O(q)$-structure $E$. Pick $\{ v_1, \ldots, v_{q(q-1)/2} \}$ a basis of $\mathfrak{o}(q)$ and $\{ u_1, \ldots, u_q \}$ a basis of $\R^q$. The relations:
    \begin{align*}
        \omega(\overline{v_j}) = v_i, & \quad \Theta_T(\overline{v_j}) = 0, j = \overline{1, q(q-1)/2} \\
        \omega(\overline{u_i}) = 0, & \quad \Theta_T(\overline{u_i}) = u_i, i = \overline{1, q}
    \end{align*}
    define global vector fields ${\overline{v_j}, \overline{u_i}}$ on $E$, which trivialise the transverse fibre bundle of the foliation $(E, \mathcal{F}_T\rvert_E)$ and, moreover, determine a transversal $\{1\}$-structure for $\mathcal{F}_T\rvert_E$.
    
    In particular, $\mathcal{F}_T\rvert_E$ is transversally parallelisable.
\end{theorem}

One of the main utilities of transversally parallelisable foliations is that they are "close" to being simple foliations, in the sense that when taking the closure of their leaves, we obtain a simple foliation.

\begin{proposition_definition}\label{def::basic_foliation}
    Let $(M, \mathcal{F})$ be a transversally parallelisable foliated manifold. The set of all vector fields which vanish when applied to any basic function for
    ${\mathcal{F}}$
    determines an integrable distribution on $M$, whose associated foliation is called the \textit{basic foliation of $(M, \mathcal{F}$)}, denoted by $\mathcal{F}_b$.
\end{proposition_definition}

\begin{theorem}{\cite[Theorem~4.2]{molino_riemannian_foliations_book}}
    \label{thm::basic_foliation_of_transversally_parallelisable_is_simple}
    Let $(M, \mathcal{F})$ be a transversally parallelisable foliation. Then the basic associated foliation, $\mathcal{F}_b$, is a simple foliation.
    
    Moreover, on each fibre of the smooth submersion $\pi_b: M \rightarrow M/\mathcal{F}_b$, the restriction of $\mathcal{F}$ is a foliation with dense leaves; in particular, the leaves of the foliation $\mathcal{F}$ are dense in those of $\mathcal{F}_b$.
\end{theorem}

\subsection{The basic Laplacian and homologically orientable foliations}
Let $(M, \mathcal{F})$ be a Riemannian foliation of codimension $q$. Let $g_M$ be a Riemannian metric on $M$ which is bundle-like with respect to $\mathcal{F}$. Denote by $g_T$ its associated transverse metric
${g_T \in \Symm^2_\bas(M)}$
as in \cref{rmk::riemannian_foliation_equivalent_defs}. 

When acting with the Laplacian of $g_M$ on basic functions we do not obtain basic functions; the caveat is that the adjoint of $d$ does not send basic functions to basic functions. This can be remedied with a construction of the $L^2_\bas$-adjoint of $d$ which ensures the smoothness of the results of its application to smooth functions. We briefly recall this construction.

Let $p: E\rightarrow M$ be a transverse $O(q)$-structure corresponding to $g_T$. Take 
${g:= p^*(g_M) + g_{O(q)}}$ 
where $g_{O(q)}$ is any left invariant metric on the fibers of $p$ with respect to which these fibers have volume $1$. Denote by $L_{\mathcal{F}_{T,b}}(x)$ the leaf of $\mathcal{F}_{T,b}$
(the basic fibration of $\mathcal{F}_T$, defined in \cref{def::basic_foliation})
passing through the point $x \in E$. 
Being the closures of the leaves of $\mathcal{F}_T$, the leaves of $\mathcal{F}_{T,b}$ are compact. 
This ensures that for any $f \in \mathcal{C}^\infty(E)$, the following function takes finite values:
\[
    A(f)(x) := \frac{
        \int_{L_{\mathcal{F}_{T,b}}(x)} f 
            d\vol_{g\rvert_{L_{\mathcal{F}_{T,b}}(x)}}
    }{
        \vol^{g} (L_{\mathcal{F}_{T,b}}(x))
    }
\]

In fact, much more can be said about the map $A$, which plays a pivotal role. 

\begin{theorem}[\cite{park_richardson_basic_laplacian_of_Riemannian_foliation}]
\begin{enumerate}
    \item \cite{park_richardson_basic_laplacian_of_Riemannian_foliation}[Proposition~1.1]
    The image of the map $A$ is in $\mathcal{C}^\infty_\bas(E; \mathcal{F}_T)$.
\end{enumerate}
Denote by $\eta: L^2(\Omega^*(E)) \rightarrow L^2(\Omega^*(M))$ the $L^2$ adjoint of the pullback map $p^*: L^2(\Omega^*(M)) \rightarrow L^2(\Omega^*(E))$. 

Set $P:= \eta A p^*$, the \textit{basic projection} on $\mathcal{C}^\infty(M)$. Then:
\begin{enumerate}
    \setcounter{enumi}{1}
    \item \cite{park_richardson_basic_laplacian_of_Riemannian_foliation}[Proposition~1.5,(3)]
    $P(\mathcal{C}^\infty(M)) \subset \mathcal{C}^\infty_\bas(M;\mathcal{F})$.
\end{enumerate}

Denote by $L^2(\mathcal{C}^\infty_\bas(M;\mathcal{F}))$ the closure of $\mathcal{C}^\infty_\bas(M;\mathcal{F})$ with respect to the inner product induced by $g_M$. Then:
\begin{enumerate}
    \setcounter{enumi}{2}
    \item \cite{park_richardson_basic_laplacian_of_Riemannian_foliation}[Proposition~1.6]
    For $f \in \mathcal{C}^\infty(M)$, $P(f)$ is the orthogonal projection of $f$ on $L^2(\mathcal{C}^\infty_\bas(M;\mathcal{F}))$.
\end{enumerate}

Let $\delta$ be the Riemannian adjoint of $d$ with respect to $g_M$. Then
\begin{enumerate}
    \setcounter{enumi}{3}
    \item 
    $P\delta$ is the adjoint of $d$ in the space $L^2(\mathcal{C}^\infty_\bas(M;\mathcal{F}))$ with respect to the inner product induced by $g_M$.
\end{enumerate}
\end{theorem}

On the other hand, one is naturally inclined to consider the adjoint of $d$ on local leaf spaces. More precisely, the fact that $g_T$ induces a volume form $d\vol_T$ on any local leaf space which happens to be a manifold permits the definition of a transverse Hodge star operator for any $r \in \{ 0, \ldots, q \}$:
\[
    *_T: \Omega^{r}_\bas(M;\mathcal{F}) \rightarrow \Omega^{q-r}_\bas(M; \mathcal{F})
\]
is defined via 
\[
    \alpha \wedge *_T \beta = \frac{1}{r!} g_T(\alpha, \beta) d\vol_T, 
    \quad \forall \alpha, \beta \in\Omega^{r}_\bas(M;\mathcal{F})
\]
Thus there exists a transverse codifferential 
$\delta_T: \Omega^{r+1}_\bas(M;\mathcal{F}) \rightarrow \Omega^{r}_\bas(M;\mathcal{F})$ defined as:
\begin{equation}
\label{eqn::basic_codiff}
    \delta_T = -(-1)^{q \cdot r} *_T d *_T
\end{equation}

One is naturally inclined to compare $\delta_T$ with $\delta_\bas:=P\delta$.
In general, they differ. However, their difference is controlled - in an explicit but complicated way - by the basic component of the mean curvature of the foliation. For our purposes, we note only:

\begin{theorem}
[Immediate~consequence~of~{\cite[Corollary~3.3]{alvarez_lopez_1992_basic_component_of_mean_curv}}]
\label{thm::general_formula_difference_basic_codiff_and_transverse_codiff}
Let $\mathcal{F}$ be an oriented and transversally oriented Riemannian foliation on the smooth compact manifold $M$. Then the difference:
\begin{equation*}
    \delta_\bas - \delta_T
\end{equation*}
is a zero order differential operator.
\end{theorem}

Finally, we recall the following definition:

\begin{definition}
\label{def::homologically_orientable_foliations}
    A smooth foliation $(M, \mathcal{F})$ is called \textit{homologically orientable} if $H^{\codim(\mathcal{F})}_{\bas}(M;\mathcal{F}) \neq 0$. 
\end{definition}

\section{Transversally K\"ahler foliations}
\label{sec::transversally_kaheler_foliatons}
Henceforth, if $(M, \mathcal{F})$ is a foliated manifold and $T: TM^{\otimes k} \rightarrow \R$ is a tensor on $M$ which vanishes on $T\mathcal{F}$ on any component, we will denote by $\overline{T}$ the induced, well-defined tensor on $(\nu \mathcal{F})^{\otimes k}$.

Let $(M, \mathcal{F})$ be a Riemannian foliation and let $g$ be the symmetric, semi positive definite metric induced from a reduction of the transverse frame bundle $\nu F$ to $O(q)$. Suppose also there exists $\omega_0 \in \Omega^2_{\bas} (M, \mathcal{F})$ which is nondenegerate on $\nu F$.

\begin{definition}\label{def::transversally_kaehler_foliation}
    In the above context, if
    \begin{itemize}
        \item $\overline{J}:= \overline{\omega_0}^{-1} \circ \overline{g}: \nu \mathcal{F} \rightarrow \nu \mathcal{F}$
        \footnote{
            This notation is in fact a shorthand; what it actually means is that 
            \[ 
                \overline{J} = (\sqrt{AA^*})^{-1} A,
            \] where, as in
            \cite[Proposition~12.3]{da_silva_symplectic_geometry_book}, $A$ is a linear map satisfying 
            \[
                \overline{\omega_0}(\cdot, \cdot) = \overline{g}(A \cdot, \cdot).
            \]
        }
        is an integrable almost complex structure on any local leaf space which happens to be a smooth manifold, and
        \item $\omega_0$ is closed ($d\omega_0 = 0$),
    \end{itemize} we call $(M, \mathcal{F}, g, \omega_0)$ a \textit{transversally K\"ahler} foliation.
\end{definition}

\begin{remark}\label{rmk::riemannian_foliation_hermitian_is_transversally_kaehler_under_conditions}
    Suppose $(M, \mathcal{F})$ is a Riemannian foliation and let $g$ be induced from a reduction of $\nu \mathcal{F}$ to $O(q)$ as before. Suppose that $M$ itself admits a complex structure $J$ invariating $g$ and such that ${T \mathcal{F}}$ is $J$-invariant. Consider $\omega_0 \in \Omega^2(M)$ defined by ${\omega_0(\cdot, \cdot) := g (J\cdot, \cdot)}$. Then $\omega_0$ is basic, vanishes on $T\mathcal{F}$ and is non-degenerate on $\nu \mathcal{F}$. Therefore, a particular case of a transversally K\"ahler foliation is the data of a Riemannian foliation and a complex structure invariating $T\mathcal{F}$ and compatible with the induced pseudometric, such that the associated Hermitian form is closed.

    Moreover, in this case, $J$ clearly descends to $\overline{J} \in \Gamma(\End(\nu \mathcal{F}))$ of square $-\overline{\Id}$. The endomorphism $\overline{\omega_0}^{-1} \circ \overline{g}$ as appearing in \cref{def::transversally_kaehler_foliation} is by definition the orthogonal part in the polar decomposition of the (invertible) endomorphism $A$ of $\nu \mathcal{F}$, defined by $\overline{\omega_0}(\cdot, \cdot) = \overline{g}(A \cdot, \cdot)$. $A$ is unique, because of the injectivity of the mapping 
    ${u \in \nu \mathcal{F} \mapsto \overline{g}(u, \cdot) \in (\nu \mathcal{F})^*}$; 
    moreover, $A$ is invertible by the nondegeneracy of $\overline{\omega_0}$, and thus its polar decomposition is unique. But since $A$ is unique, it must coincide with $\overline{J}$, and since the polar decomposition of $A$ is unique and $\overline{J}$ is already orthogonal, it follows that $\overline{J}$ is the orthogonal part in the polar decomposition of $A$, so there is no danger of notational ambiguity between the $\overline{J}$ induced from $\overline{\omega_0}$ and $\overline{g}$ on the one hand, and the $\overline{J}$ induced from the complex structure $J \in \End(TM)$ on the other.
\end{remark}

Suppose 
$(M, \mathcal{F}, \omega_0, g)$ is a transversally K\"ahler foliation with
$\overline{J} \in \End(\nu \mathcal{F})$ as in \cref{def::transversally_kaehler_foliation}. Since $\overline{J}$ is an almost complex structure, it induces a direct sum decomposition 
$\nu \mathcal{F} = \nu \mathcal{F}^{1,0} \oplus \nu \mathcal{F}^{0,1}$ corresponding to the eigenvalues $\sqrt{-1}$ and $-\sqrt{-1}$ of $\overline{J}$, respectively. This induces a direct sum decomposition at the level of exterior algebras in the standard way.

The basic differential
$
d: 
\Omega^*_{\bas}(M; \mathcal{F}) 
    \rightarrow 
\Omega^{*+1}_{\bas}(M; \mathcal{F})
$
therefore splits as 
$d = \6 + \bar\6$ where
$
\6: \Omega^*_{\bas}(M; \mathcal{F}) 
    \rightarrow 
\Omega^{*+1,*}_{\bas}(M; \mathcal{F})
$
and
$\newline
{
\bar\6: \Omega^{*,*}_{\bas}(M; \mathcal{F}) 
    \rightarrow 
\Omega^{*,*+1}_{\bas}(M; \mathcal{F})
}
$. 

The following essential fact is an analogue of what happens in the K\"ahler situation.

\begin{lemma}[{\cite[Proposition~3.5.1]{alaoui_1990_operateurs_transversalement_elliptiques}}]
\label{lemma::transversal_global_del_del_bar}
    Let $(M, \mathcal{F})$ be a foliated manifold where $\mathcal{F}$ is a homologically orientable and transversally K\"ahler foliation. Suppose 
    $\omega, \omega' \in \Omega^{1,1}_{\bas}(M; \mathcal{F})$ 
    are such that $[\omega] = [\omega']$ in 
    $H^2_{\bas}(M; \mathcal{F})$. Then there exists 
    $f \in \mathcal{C}^\infty_{\bas}(M; \mathcal{F})$
    such that 
    $\omega' = \omega + \sqrt{-1} \6_T\bar\6_T f$.
\end{lemma}

Just like in the K\"ahler situation, \cref{lemma::transversal_global_del_del_bar} is a consequence of a transversal version of the Dolbeault decomposition, which occurs in the presence of homological orientability and is one of the main topics of \cite{alaoui_1990_operateurs_transversalement_elliptiques}.

\subsection{Riemannian properties of transversally K\"ahler foliations}

Consider now $\nabla: \Gamma(\nu \mathcal F) \rightarrow \Gamma(\nu \mathcal F) \otimes \Omega^1(M)$ the Levi-Civita connection of a Riemannian metric $\overline{g}$ on $\nu \mathcal F$. 

Consider the curvature tensor ${R^\nabla: \Gamma(TM) \times \Gamma(TM) \times \Gamma(\nu \mathcal F) \rightarrow \Gamma(\nu \mathcal F)}$ and its trace ${\overline{\Ric}: \Gamma(\nu \mathcal F) \times \Gamma(TM) \rightarrow \R}$, ${\overline{\Ric} (Z, Y) := \tr (X \mapsto R^\nabla(X,Y)Z )}$.

A natural definition in this context is that of a transversally Einstein foliation. Denoting by ${\pi: TM \rightarrow \nu \mathcal{F}}$ the projection, we can make:

\begin{definition}\label{def::transversally_einstein_foliation}
    A Riemannian foliation $(M, \mathcal{F})$ with Riemannian metric $\tilde g$ on $\nu F$ is called \textit{transversally Einstein} if:
    \[
        \overline{\Ric}(\cdot, \cdot) = \lambda \overline{g}(\cdot, \pi \cdot )
    \]
    for some $\lambda \in \R$ called \textit{the Einstein constant}.
\end{definition}

Now, if $(M, \mathcal{F}, \omega_0, g, \overline{J})$ is transversally K\"ahler, then $\overline{\rho}^{\omega_0}(\cdot, \cdot):= \overline{\Ric}(\overline{J} \cdot, \cdot)$ defines a skew-symmetric $2$-form on $\nu \mathcal{F}$, just as in the usual K\"ahler case (\cite[Section~12.1]{moroianu_lectures_kaehler})

The extension of the transverse Ricci form $\overline{\rho}^{\omega_0}$ to $TM$ behaves like a transversally K\"ahler form but without the nondegeneracy. More precisely, we have:
\begin{proposition}\label{prop::transversal_Ricci_curvature_well_def}
    Consider $\rho^{\omega_0}$ to be the natural extension of $\overline{\rho}^{\omega_0}$ to ${\Gamma(TM) \times \Gamma(TM)}$, i.e. $\rho^{\omega_0} = \overline{\rho}^{\omega_0}(\pi \cdot , \pi \cdot)$, where $\pi: TM \rightarrow \nu \mathcal F$ is the projection. Then $\rho^{\omega_0}$ is a basic, closed, real $(1,1)$-form on $M$.
\end{proposition}

\begin{definition}
\label{def::transversal_Ricci_form}
    $\rho^{\omega_0}$ will be called the \textit{transversal Ricci form} associated to the trasnvserally K\"ahler metric $\omega_0$.
\end{definition}

As a further analogue of the classical K\"ahler case, we have:
\begin{proposition}
\label{prop::transversal_Ricci_curvature_def_chern_class}
    Let $(M, \mathcal{F}, \omega_0)$ be a transversally K\"ahler foliation. Then the cohomology class $[\rho^{\omega_0}] \in H^2_\bas(M; \mathcal{F})$ is mapped via the natural morphism 
    $H^2_\bas(M;\mathcal{F}) \rightarrow H^2(M)$
    to $c_1(\nu \mathcal{F}) \in H^2(M)$.
\end{proposition}

Thus, the following definition is also natural.
\begin{definition}\label{def::positive_negative_transversally_kaehler_foliation}
    Let $(M, \mathcal{F})$ be a transversally K\"ahler foliation. $(M, \mathcal{F})$ is called \textit{transversally negative} (\textit{positive}) (with respect to the fixed transversally holomorphic structure)
    if $c_1(\nu \mathcal{F})$ is the image of a class $c_B \in H^2_\bas(M;\mathcal{F})$ via the natural morphism 
    $H^2_\bas(M;\mathcal{F}) \rightarrow H^2(M)$ 
    such that a representative $\omega$ of $c_B$ is transversally negative (positive) i.e. the symmetric tensor on $\nu \mathcal{F}$ 
    \[
        \overline{\omega}(\cdot, \overline{J} \cdot),
    \]
    is negative (positive)-definite.
\end{definition}

\begin{remark}
    Like in the usual K\"ahler case, if $(M, \mathcal{F}, \omega_0, g)$ is a transversally K\"ahler foliation, then $\mathcal{F}$ is transversally Einstein if and only if:
    \[
        \overline{\rho}^{\omega_0}(\cdot, \cdot) = \lambda \omega_0(\pi \cdot, \pi \cdot)
    \]
    for some $\lambda \in \R$, the same Einstein constant as before.
\end{remark}

\section{Elliptic operators on transversally K\"ahler foliations}
\label{sec::elliptic_operators_on_transv_kaehler}
\subsection{Local theory of elliptic operators}
In this section we recall the definition of H\"older spaces and a result about operators defined on an open subset ${U \subset \R^m}$. We assume throughout that $U$ is bounded and that $\6 U$ is a smooth submanifold, although these conditions can be relaxed.

In the following definition, $I$ will denote a multi-index of the form $(i_1, \ldots, i_l)$ - of length $l$ in this case, which will be denoted by $\lvert I \rvert$.
For any open subset $U \subset \R^m$ and $u \in \mathcal{C}^\infty(M)$, and any $k \in \N$ and $\alpha \in (0,1)$, the \textit{$(k,\alpha)$-H\"older norm on $U$} is defined as:
\begin{align*}
    \lVert u \rVert_{k, \alpha, U} &:= \lVert u \rVert_{k, U} + \left[ D^k u \right]_{\alpha, U}, \quad \textrm{where} \\
    \lVert u \rVert_{k, U} &:= \sum_{i=0}^k \lVert D^i u \rVert_{L^\infty(U)} 
    \, \, \textrm{with the notation} \, \, 
    \lVert D^i u \rVert_{L^\infty(U)} := \max_{|I| = i} |D^I u|_{L^\infty(U)}
    \\
    \textrm{and} &\left[ D^k u \right]_{\alpha, U} := \sup_{\lvert I \rvert = k} \sup_{x \neq y \in U} 
                            \frac{
                                \lvert D^I u (x) - D^I u (y) \rvert
                            }{
                                \lvert x- y \rvert^\alpha
                            }
\end{align*}

\begin{definition}
    Let
    \[
        \mathcal{C}^{k, \alpha}(U) := \{ u \in \mathcal{C}^k(U): \lVert u \rVert_{k,\alpha, U} < \infty \}
    \]
\end{definition}

We now recall the following definition, fixing a class of operators having a degree of regularity as relaxed as can be imposed to obtain meaningful results.

\begin{definition}\label{def::nonlinear_elliptic_operator}
    Let $\mathcal{L}$ be a real, twice continuously differentiable function on the set ${U \times \R \times \R^m \times \Sym_{m \times m}(\R)}$. We see the set of symmetric matrices $\Sym_{m \times m}(\R)$ as a subset of $M_{m \times m}(\R)$ and denote by $r_{i,j}$ the $n^2$ coordinates in $M_{m \times m}(\R)$. Let $u \in \mathcal{C}^2(U)$. We say $\mathcal{L}$ is
     \textit{elliptic with respect to $u$} if the $m \times m$ matrix given by:
            \[
                \left( \6_{r_i} \6_{r_j} \mathcal{L} \right)_{i,j}
            \]
            is positive definite at all points in the image of the mapping
            \[
                U \ni x \mapsto \left( x, u(x), Du(x), D^2u(x) \right)
            \]
\end{definition}

\begin{theorem}[Regularity of solutions; {\cite[Lemma~17.16]{gilbarg_trudinger_elliptic_pdes_book}}]
\label{thm::regularity_for_fully_nonlinear}
    Let 
    \[ 
        {\mathcal{L}: \mathcal{C}^2(U) \rightarrow \mathcal{C}^0(U)}
    \]be a (nonlinear) second order operator and $u \in \mathcal{C}^2(U)$ be elliptic with respect to $\mathcal{L}$ and satisfying $\mathcal{L}(u) = 0$ on $U$. If the coefficients of $\mathcal{L}$ are of class $\mathcal{C}^{k, \alpha}(U)$, then $u \in \mathcal{C}^{k+2, \alpha}(U)$.
\end{theorem}

\subsection{Elliptic operators on smooth foliations}

Let $(M, \mathcal{F})$ be a smooth foliation of $\codim (\mathcal{F}) = m$. Choose an at most countable foliated atlas $(U_i, \varphi_i)_{i}$, which will be assumed finite if the manifold is compact, and which will be fixed once and for all. Denote by $T_i$ the quotient manifolds $T_i:= U_i / (\mathcal{F}\rvert_{U_i})$.

Let $u \in \mathcal{C}^k_\bas(M; \mathcal{F})$. By  (e.g.) \cite{molino_riemannian_foliations_book}[Proposition~2.1], $u\rvert_{U_i}$ can be seen as a function on $T_i$, which we will denote $\overline{u_i} \in \mathcal{C}^k(T_i)$.
We can define a global $(k,\alpha)$-H\"older norm on $\mathcal{C}^k_{\bas}(M; \mathcal{F})$ by considering the supremum of the local norms on each foliated chart:
\begin{definition}
    Let $u \in \mathcal{C}^k_{\bas}(M; \mathcal{F})$. Set:
    \[
        \lVert u \rVert_{k, \alpha} := \sup_{i} \lVert 
                    \overline{u_i}
        \rVert_{k, \alpha, T_i}
    \]
    We also denote ${
        \mathcal{C}^{k, \alpha}_\bas (M; \mathcal{F}) := \left\lbrace u \in \mathcal{C}^{k}_\bas (M; \mathcal{F}) : \lVert u \rVert_{k, \alpha} < \infty \right\rbrace
    }$
\end{definition}

Consider also the maps $\pi_i:= \pr_{\R^m} \circ \varphi_i: U_i \rightarrow \R^m$.

In the discussion which follows, we will require the following result, which highlights one of the pivotal roles that H\"older spaces play in the argument: on compact manifolds, boundedness in H\"older spaces guarantees, up to a decrease in the H\"older exponent, convergence in these spaces. 
\begin{proposition}\label{prop::compact_embedding_Holder_spaces}
    Let $(M, \mathcal{F})$ be a smoothly foliated compact manifold and $k \in \N$, as well as $\alpha, \beta \in (0,1)$ with $\beta > \alpha$.  Then the inclusion 
    \[ 
        {\mathcal{C}_\bas^{k,\beta}(M; \mathcal{F}) \hookrightarrow \mathcal{C}^{k, \alpha}_\bas (M; \mathcal F)}
    \]
    is a compact operator i.e. sends bounded sets to precompact ones.
    \begin{proof}
        Let $\{u^p\}_{p \in \N}$ be a bounded sequence of functions in $\mathcal{C}_\bas^{k,\beta}(M; \mathcal{F})$. The hypothesis translates to the existence of a $C>0$ such that on any foliated chart $U_i$ in our fixed foliated atlas and any $p \in \N$:
        \[
           \sum_{j=0}^k \lVert D^j \overline{u_i^p} \rVert_{L^\infty(T_i)} 
           + \sup_{\lvert I \rvert = k} \sup_{x \neq y \in T_i} 
                            \frac{
                                \lvert D^I \overline{u_i^p} (x) - D^I \overline{u_i^p} (y) \rvert
                            }{
                                \lvert x- y \rvert^\beta
                            } \leq C
        \]
        In particular for any multi-index $I$ with $\lvert I \rvert = k$ we have that $\{ D^I \overline{u_i^p} \}_{p \in \N}$ is bounded in $L^\infty(T_i)$. Because:
        \[
                             \sup_{x \neq y \in T_i}   \frac{
                                \lvert D^I \overline{u_i^p} (x) - D^I \overline{u_i^p} (y) \rvert
                            }{
                                \lvert x- y \rvert^\beta
                            } \leq C,
        \]
        $\{ D^I \overline{u_i^p} \}_{p \in \N}$ is also equicontinuous. Thus the Arzel\`a-Ascoli theorem applied to the sequence $\{ D^I \overline{u_i^p} \}_{p \in \N}$ guarantees the existence of a subsequence which converges uniformly in the compact-open topology on the space of continuous functions. Because $M$ is assumed compact, this reduces to global uniform convergence. We can thus assume that the subsequence is the whole $\{ D^I \overline{u_i^p} \}_{p \in \N}$, and then apply the same reasoning for as many times as there are multiindeces $I$ with $\lvert I \rvert = k$ to be able to assume that for \textit{any} $I$ with $\lvert I \rvert = k$ we have uniform convergence of $ D^I \overline{u_i^p} $ to a function denoted $u_i^I$.

        Now, for $J$ a multiindex with $\lvert J \rvert = k-1$, ${ \{D^J \overline{u_i^p} \}_{p \in \N} }$ is bounded, so, passing to a subsequence, converges pointwise in at least a point. Therefore passing to subsequences again we can assume that for any $J$ with $\lvert J \rvert = k-1$ we find functions $u_i^J$ such that $ \{D^J \overline{u_i^p} \}_{p \in \N}$ converges to $u_i^J$ in $L^\infty(T_i)$. Continuing this process inductively, we arrive at a ${ u_i \in \mathcal{C}^k(T_i) }$ such that 
        all derivatives up to and including order $k$ of $\overline{u_i^p}$ converge to $u_i$ in $L^\infty(T_i)$.

        So $u_i \in \mathcal{C}^{k}(T_i)$ and in fact $u_i \in \mathcal{C}^{k, \beta}(T_i)$ by pointwise convergence of $k$-th order derivatives.
        
        Furthermore, since $\beta - \alpha > 0$, we have that $\lvert x-y \rvert^{\beta - \alpha}$ approaches $0$ as $x$ gets close to $y$. 
        Combined with the fact that 
        $[D^I \overline{u_i^p}]_{\alpha, T_i} < \infty$ and 
        $[D^I u_i]_{\alpha, T_i} < \infty$ for $|I| = k$, it follows that
        ${\lVert \overline{u_i^p} - u_i \rVert}_{k, \alpha, T_i} \rightarrow 0$.

        Finally, we note that the $u_i \in \mathcal{C}^{k, \alpha}(T_i)$ obtained as above glue well to a $u \in \mathcal{C}^{k,\alpha}_\bas(M; \mathcal{F})$ because for any $p \in \N$, $\pi_i^*(\overline{u_i}^p) = u^p$ on $U_i$, whence ${\pi_i^*(u_i) = f_j^* (u_j)}$ on $U_i \cap U_j$, which completes the proof.
    \end{proof}
\end{proposition}

H\"older spaces have the Sobolev fractional spaces as analogues in the weak setting. However, since a crucial role in the main argument of 
\cref{subsec::operator_cal_T_on_transv_Kahler_foliations} 
will be played by the transverse Laplacian, we will develop the necessary arguments involving the integer Sobolev spaces which give the knowledge we require of the transverse Laplacian.

\begin{definition}
    Let $M$ be a compact boundaryless manifold and $u \in L^1(M)$. 
    Recall that for $\alpha$ a multiindex, the $\alpha$-th weak derivative of $u$ is the function (unique if it exists) $D^{\alpha}u \in L^1(M)$ satisfying:
    \[
        \int_M u D^\alpha \phi d\vol = (-1)^{|\alpha|}
        \int_M (D^\alpha u) \phi d\vol, \quad \forall \phi \in \mathcal{C}^\infty(M)
    \]
    $W^{k, p}(M)$ denotes the space of all $u \in L^1(M)$ such that $D^\alpha u$ exists for any multi-index $\alpha$ with $|\alpha| \leq k$ and $D^\alpha u \in L^p(M)$.

    Recall that with respect to a certain norm $\lVert \cdot \rVert_{W^{k,p}(M)}$ defined in terms of the integrals of all weak derivatives, $W^{k,p}(M)$ are Banach spaces, and $W^{k,2}(M)$ are Hilbert spaces 
    (\cite[Chapter~7.5]{gilbarg_trudinger_elliptic_pdes_book})

    Suppose now $\mathcal{F}$ is a foliation on $M$.
    Then we define $W^{k,p}_{\bas}(M;\mathcal{F})$ to be the closure of 
    $\mathcal{C}^\infty_{\bas}(M;\mathcal{F})$ inside 
    $W^{k,p}(M)$.
\end{definition}

The proof of the following result, 
\cref{thm::transverse_laplacian_plus_identity_invertible},
takes one of its main ideas from the more complex proof of a similar statement but for a different operator - namely, $\Delta_\bas$ - which is presented in
\cite[Section~5]{kamber_tondeur_87_de_Rham_Hodge_theory_Riemann_fol}. We follow, moreover, the general approach to weak solvability, expounded e.g. in \cite[Chapter~6]{evans_2010_pdes_book}, while for the regularity of weak solutions we follow \cite{gilbarg_trudinger_elliptic_pdes_book}.
In our case, we work with the transverse Laplacian, which simplifies the proof to an extent.

\begin{theorem}
\label{thm::transverse_laplacian_plus_identity_invertible}
    Let $(M;\mathcal{F})$ be an oriented and transversally oriented Riemannian foliation with $\codim (\mathcal{F}) = q$. Let $g_M$ be a bundle-like metric on $M$. 
    Let $k \in \N$, $k > 2 + q/2$. 
    Consider the transverse Laplace-Beltrami operator
    \[
        {\Delta_T: \mathcal{C}^{k+2,\alpha}_\bas(M;\mathcal{F}) \rightarrow \mathcal{C}^{k,\alpha}_\bas(M;\mathcal{F})}
    \]
    defined as
    \[ 
        \Delta_T:= d\delta_T + \delta_T d.
    \] 
    Then the operator
    \begin{align*}
        L &: \mathcal{C}^{k+2,\alpha}_\bas(M;\mathcal{F}) \rightarrow \mathcal{C}^{k,\alpha}_\bas(M;\mathcal{F}) \\
        L &:= \Delta_T + \Id
    \end{align*}
    is an isomorphism of Banach spaces.
\end{theorem}

\begin{proof}
[Proof]
    First we show that $L$ is injective.
    Suppose $u \in \mathcal{C}^{k+2,\alpha}_\bas(M;\mathcal{F})$ satisfies $\Delta_T u + u = 0$. 
    Since $M$ is compact, there exist $x_{\min}$ and $x_{\max}$ such that $u(x_{\max}) = \max_M u$ and $u(x_{\min}) = \min_M u$. Consider $\overline{u}$ as a function on a smooth leaf space around $x_{\max}$. Since the Hessian of $\overline{u}$ is negative semidefinite in a maximum point:
    \[
        0 = (\Delta_T \overline{u})(\overline{x_{\max}}) + \max_M u 
        = - \tr (\textrm{Hess}^{\overline{u}})(\overline{x_{\max}}) + \max_M u
        \geq 0 + \max_M u
    \]
    Therefore $\max_M u \leq 0$. Analogously we see that $\min_M u \geq 0$ and thus $u = 0$.

    We now prove surjectivity.
    Using the density of $\mathcal{C}^{k+2,\alpha}_{\bas}(M;\mathcal{F})$ in $L^2_{\bas}(M;\mathcal{F})$ (with respect to the $L^2$ norm), define a bilinear form on $L^2_{\bas}(M;\mathcal{F})$ by the formula $B_L(u, v) := \langle L u, v \rangle$, where
    $\langle \cdot, \cdot \rangle$ denotes the standard inner product.
    By \cref{thm::general_formula_difference_basic_codiff_and_transverse_codiff}, 
    $\delta_{\bas} - \delta_T$ is a zero order term. Thus $\Delta_T$ differs from $\Delta_{\bas}$ by a first order term.
    By adding another first order term to $\Delta_{\bas}$ as in \cite[Section~3]{kamber_tondeur_87_de_Rham_Hodge_theory_Riemann_fol}, we can complete $L$ to a strongly elliptic operator on $\mathcal{C}^{2,\alpha}(M)$ and thus G{\aa}rding's inequality holds for $L$ 
    (the proof being akin to e.g. \cite[Section~6.2.2,~Theorem~2]{evans_2010_pdes_book}). 
    Namely, for some constants $c_1, c_2 > 0$ we have for any 
    $u \in W^{1,2}_{\bas}(M;\mathcal{F})$:
    \begin{equation}
    \label{eqn::garding_inequality}
          c_1 \lVert u \rVert_{W^{1, 2}(M)}
          \leq B_L(u,u) + c_2 \lVert u \rVert_{L^2(M)}
    \end{equation}

    Take $\gamma := c_2$. Then as a consequence of \eqref{eqn::garding_inequality}, the bilinear form $B_{L+\gamma \Id} = B_L + \gamma \langle \cdot, \cdot \rangle$ satisfies the coercivity condition of the Lax-Milgram theorem 
    (\cite[Section~6.2.1]{evans_2010_pdes_book}). In this way we obtain that for a given 
    $v \in L^2_{\bas}(M)$
    the equation
    \[
        B_{L+\gamma \Id} (u, y) = \langle v, y \rangle, \quad \forall y \in W^{1,2}_{\bas}(M;\mathcal{F})
    \]
    has a unique solution in the unknown $u \in W^{1,2}_{\bas}(M;\mathcal{F})$,
    which we denote by $G_{L+\gamma \Id}(v)$. By the analogue of the Rellich-Kondrachov compact embedding theorem for the setting of basic functions (\cite[Proposition~4.5]{kamber_tondeur_87_de_Rham_Hodge_theory_Riemann_fol}), the space $W^{1,2}_{\bas}(M;\mathcal{F})$ embeds compactly into $L^2_{\bas}(M;\mathcal{F})$, and thus the function $G_{L+\gamma \Id}$ seen as an operator from $L^2_{\bas}(M;\mathcal{F})$ to $L^2_{\bas}(M;\mathcal{F})$ is a compact operator on the Hilbert space $L^2_{\bas}(M;\mathcal{F})$. Since it is also bounded and linear, the Fredholm alternative (\cite[Theorem~5.3]{gilbarg_trudinger_elliptic_pdes_book}) applies for it (and for any scalar multiple of it) which shows that either
    $\ker \left( 
        \Id - \gamma G_{L+\gamma \Id}
    \right) = \{ 0 \}$ or
    $\Im \left( 
        \Id - \gamma G_{L+\gamma \Id}
    \right) = L^2_{\bas}(M;\mathcal{F})$.
    
    But if 
    $u \in \ker 
        \left(
            \Id -\gamma G_{L+\gamma \Id} 
        \right)$
    then by the definitions $B_{L+\gamma \Id}(u, \cdot) = 0$, hence $u \in \ker(L)$. Therefore, since $L$ is injective, 
    $\ker \left( 
        \Id - \gamma G_{L+\gamma \Id}
    \right) = \{ 0 \}$ 
    and thus the equation $Lu = f$ for prescribed $f \in L^2_{\bas}(M;\mathcal{F})$ has a unique solution $u \in L^2_{\bas}(M;\mathcal{F})$ by the remaining case of the Friedrichs alternative and a similar reasoning as the one above.
    
    Let now $f \in \mathcal{C}^{k,\alpha}_{\bas}(M;\mathcal{F})$. The last step in proving the surjectivity of $L$ is proving that the unique solution $u \in L^2_{\bas}(M;\mathcal{F})$ of $Lu = f$ has in fact a representative in $\mathcal{C}^{k+2,\alpha}_{\bas}(M;\mathcal{F})$. First note that $G_{L+\gamma \Id}$ actually takes values in $W^{1, 2}_{\bas}(M;\mathcal{F})$, so $u \in W^{1,2}_{\bas}(M;\mathcal{F})$. Now, regularity is a local property, so it is enough to show that for an arbitrary foliated chart $U$, the function $\overline{u} \in W^{1,2}(U/\mathcal{F}\rvert_U)$ is smooth. But $L$ restricts in the leaf space $U/\mathcal{F}\rvert_U$ to the elliptic operator $\Delta_T + \Id$ with smooth coefficients. Therefore, \cite[Corollary~8.10]{gilbarg_trudinger_elliptic_pdes_book} applies directly, showing that for any relatively compact open subdomain $U' \subset U$, $\overline{u} \in W^{k, 2}(U')$. Since $k> 2 + q/2$, from \cite[Corollary~7.11]{gilbarg_trudinger_elliptic_pdes_book} it follows that $\overline{u} \in C^2(U)$. Thus, by \cite[Theorem~6.17]{gilbarg_trudinger_elliptic_pdes_book} it follows that $u \in \mathcal{C}^{k,\alpha}(U)$, which in turn shows that $u \in \mathcal{C}^{k, \alpha}_{\bas}(M; \mathcal{F})$

    Now we note that by the Banach Open Mapping Theorem (e.g. 
    \cite[Theorem~2.11]{rudin_functional_analysis}), it is enough to prove that the bounded operator $L$ is bijective to show that $L$ has bounded inverse, and thus is an isomorphism of Banach spaces, concluding the proof.
    
\end{proof}

\subsection{The operator $\Cal_T^-$ on transversally K\"ahler foliations}
\label{subsec::operator_cal_T_on_transv_Kahler_foliations}

Assume from now on that $(M, \mathcal{F}, \omega_0, g_0)$ is a transversally K\"ahler foliation of $\codim \mathcal{F} = m = 2q$. Denote by $\6_T$ and $\bar\6_T$ the transverse Dolbeault operators with respect to the complex complex structures supposed in \cref{def::transversally_kaehler_foliation}. Denote by $\overline{\omega_{0,i}}$ the smooth $2$-form obtained from $\omega_0$ on $T_i$.
Then there is a well defined operator
$\Cal^{-}_T$ on $\mathcal{C}^{\infty}_{\bas} (M; \mathcal{F})$, which we arrive at naturally in the proof of \cref{thm::transversal_aubin_yau}.
Namely, for $u \in \mathcal{C}^{\infty}_\bas(M;\mathcal{F})$, whenever $x \in U_i$ we have the following expression of $\Cal^{-}_T (u)(x)$:
\[
        \Cal^{-}_T (u)(x) := 
        \left[ 
                \log 
                \left(
                    \frac{ 
                            (
                                \overline{\omega_{0,i}} + 
                                \sqrt{-1} \6_{T} \bar\6_{T} (\overline{u_i}) 
                            )^q 
                       }
                        {
                            (\overline{\omega_{0,i}})^q
                        }
                \right)
                - \overline{u_i} 
        \right]  (\pi_i(x))
\]

Since this expression is independent of the chart, we will write
\[
    \Cal^{-}_T (u) := 
                \log 
            \left(
                \frac{ 
                        (
                            \overline{\omega_0} + 
                            \sqrt{-1} \6_{T} \bar\6_{T} \overline{u} 
                        )^q 
                   }
                    {
                        (\overline{\omega_0})^q
                    }
            \right)
            - \overline{u} 
\]

The goal of the remainder of this section is to justify the following theorem, which is what our main result \cref{thm::transversal_aubin_yau} reduces to.
\begin{theorem}
\label{thm::transverse_calabi_yau_equation_has_solutions}
Let $(M, \mathcal{F})$ be a transversally K\"ahler.
Let $f \in \mathcal{C}^{\infty}_\bas(M; \mathcal{F})$. 
Then the equation
    \[
        \Cal^{-}_T (u) = f
    \]
    has a solution $u \in \mathcal{C}^{\infty}_\bas(M; \mathcal{F})$
\end{theorem}

We shall need

\begin{lemma}
\label{lemma::cal_is_elliptic_wrt_u}
    Let $u \in \mathcal{C}^{2}_\bas(M;\mathcal{F})$
    be such that
    ${\overline{\omega_0} + \6_T\bar\6_T \overline{u} > 0}$.
    Then $\Cal^{-}_T$ is elliptic with respect to $u$.
    \begin{proof}
        Ellipticity is a local property, so we can assume $u$ is defined on the leaf space of a foliated chart and work there. The relevant (order $2$) component of our operator in its local expression is:
        \[
            \mathcal{L}_2 (u) := \log \det (\6_i \bar\6_j u + g_{i,\bar j}) )
        \]
        We identify complex matrices with real matrices in the standard way: 
        \[
            \C^{q \times q} \ni M 
            \mapsto 
                Z(M):= \begin{pmatrix}
                    \Re (M) & - \Im (M) \\
                    \Im (M) & \Re (M)
                \end{pmatrix}
                \in \R^{2q \times 2q}
        \]

        The coordinates on $\R^{2q \times 2q}$ which are compatible with this identification are
        \[
        r_{i,j} = 
        \begin{cases}
            \6_{x_i} \6_{x_j}, \quad i, j \leq q \\
            \6_{x_i} \6_{y_j}, \quad i \leq q, j > q \\
            \6_{y_i} \6_{x_j}, \quad i > q, j \leq q \\
            \6_{y_i} \6_{y_j}, \quad i, j \geq q \\
        \end{cases}
        \]
        Denote 
        $
        {M:= \left( \6_i \bar\6_j + g_{i,\bar j} \right)_{i,j}}
        $.
        On the image of $u$, $M$ as well as $M + \left( g_{i,\bar j} \right)_{i,j}$ are positive definite Hermitian matrices. Then $\det(M) = \sqrt{\det Z(M)}$. Combining this fact with the Jacobi formula applied for the invertible matrix $Z(M)$ we obtain:
        \[
            \6_{r_{i,j}} \mathcal{L}_2 
            =
                \frac{
                    \sqrt{\det(Z(M))}
                }{
                    \det \left(M + \left( g_{i,\bar j} \right)_{i,j} \right)
                } \tr \left( Z(M)^{-1} \6_{r_{i,j}} Z(M) \right)
        \]
        By computation we obtain:
        \[
            \tr \left( Z(M)^{-1} \6_{r_{i,j}} Z(M) \right)
            =
            \begin{cases}
                2\Re (M^{-1})_{j,i}, \quad i,j > q \textrm{ or } i, j <q \\
                2\Im (M^{-1})_{j,i}, \quad i>q, j \leq q \\
                -2\Im (M^{-1})_{j,i}, \quad i\leq q, j > q \\
            \end{cases}
        \]
        Therefore,
        \[ 
            \left(
            \6_{r_{i,j}} \mathcal{L}_2 
            \right)_{i,j}
            =
                \frac{
                    2\sqrt{\det(Z(M))}
                }{
                    \det \left(M + \left( g_{i,\bar j} \right)_{i,j} \right)
                }
                Z \left( 
                    \left(
                        M^{-1} 
                    \right)^T 
                \right)
        \]
        Thus, since 
        $ 
        Z \left( 
            \left(
                M^{-1} 
            \right)^T 
        \right)$ is positive definite on the image of
        \[
            x \mapsto (x, u(x), D^1 u(x), D^2 u(x)),
        \]
        so is $(\6_{r_{i,j}} \mathcal{L}_2)_{i,j}$ on this image, ending the proof.
    \end{proof}
\end{lemma}

We employ the Schauder continuity method. For each $t \in [0,1]$, consider the modified equations:
\begin{equation}\label{eqn::parametrised_equation}
        \Cal^{-}_T(u) = tf, 
    \quad 
        \overline{\omega_0} + \6_T\bar\6_T \overline{u} > 0
    \tag{$*_t$}
\end{equation}

We show that the set of $t$'s for which \hyperref[eqn::parametrised_equation]{$(*_{\bar t})$} has a solution is both open and closed in $[0,1]$ to arrive at the fact that the initial equation does so as well.

\begin{proposition}
\label{prop::solution_set_open}
    Let $(M, \mathcal{F})$ be a transversally K\"ahler foliation.
    Suppose for some $t_0 \in [0,1]$, there exists $u^{t_0} \in \mathcal{C}^\infty_\bas(M; \mathcal{F})$ which is a solution of \hyperref[eqn::parametrised_equation]{$(*_{t_0})$}. Then there is a small $\varepsilon > 0$ such that for any $s$ with ${\lvert s - t_0 \rvert < \varepsilon}$, 
    there exists a solution $u^s \in \mathcal{C}^\infty_\bas(M; \mathcal{F})$ of equation \hyperref[eqn::parametrised_equation]{$(*_s)$}.
\end{proposition}

For the proof of \cref{prop::solution_set_open}, we employ a classical theorem in infinite-dimensional analysis:
\begin{theorem}
[{\cite{dieudonne_modern_functional_analysis}}]
\label{thm::inverse_function_theorem_banach_spaces}
    Let $B_1, B_2$ be Banach spaces. Suppose the map
    \[
        \widetilde{\mathcal{L}}: B_1 \times [0,1] \rightarrow B_2
    \]
    is continuously differentiable at the point ${(u^{t_0}, t_0) \in B_1 \times [0,1]}$ and moreover the partial Fr\`echet derivative with respect to $B_1$,
    \[
        d^1_{(u^{t_0}, t_0)}{\widetilde{\mathcal{L}}}: B_1 \rightarrow B_2
    \]
    is an invertible linear map of Banach spaces i.e. it has a bounded inverse.

    Then if $\widetilde{\mathcal{L}}(u^{t_0}, t_0) = 0$, then there exists $\varepsilon > 0$ such that for any $t \in [0,1]$ with $\lvert t - t_0 \rvert < \varepsilon$ there exists a solution $u^t \in B_1$ of the equation $\widetilde{\mathcal{L}}(u, t) = 0$.
\end{theorem}

\begin{proof}[Proof of \cref{prop::solution_set_open}]
    Consider the Banach manifolds 
    \[
        B_1: = \{ u \in \mathcal{C}^{k+2,\alpha}_\bas(M;\mathcal{F}) : 
        \overline{\omega_0} + \sqrt{-1} \6_T\bar\6_T \overline{u} > 0 \}
    \]
    and
    $B_2 := \mathcal{C}^{k,\alpha}_\bas(M;\mathcal{F})$
    We seek to apply \cref{thm::inverse_function_theorem_banach_spaces} to the operator:
    \[
        \Cal^{-}_T : B_1 \rightarrow B_2
    \]
    
    Because $B_1$ is defined by an open condition, it is open in $\mathcal{C}^{k+2,\alpha}_\bas(M;\mathcal{F})$.
    Consider then
    ${
        v \in T_{u^{t_0}} B_1 
    \simeq 
        \mathcal{C}^{k,\alpha}_\bas(M;\mathcal{F})
    }$. 
    
    We compute the Fr\`echet differential of $\Cal^{-}_T$ applied on the vector $v$, using the classical definition involving a curve tangent to $v$. 
    
    Because $\Cal^{-}_T$ is a local operator, this computation can be done in a foliated chart $U_i$.
    \begin{align}
        \left(  d^1_{(u^{t_0}, t_0)} \Cal^{-}_T \right) (\overline{v})
        &=
        \left. \frac{d}{ds} \right\rvert_{s=0} \Cal^{-}_T (u^{t_0} + sv) 
        \nonumber
        \\
        &
        \stackrel{\textrm{locally}}{=}
        \left(
            \frac{ 
                (
               \overline{\omega_0} + 
               \sqrt{-1} \6_{T} \bar\6_{T} 
               (\overline{u^{t_0}}) )^{q}
           }
            {
                (\overline{\omega_0})^q
            }
        \right)^{-1}
        \nonumber
        \\ 
        &\cdot \frac{ 
                   q \left(
                       \overline{\omega_0} + 
                       \sqrt{-1} \6_{T} \bar\6_{T} 
                       (\overline{u^{t_0}}) 
                   \right)^{q-1} 
                \wedge 
                    \sqrt{-1} \6_{T} \bar\6_{T} \overline{v}
           }
            {
                \left(
                   \overline{\omega_0} + 
                   \sqrt{-1} \6_{T} \bar\6_{T} 
                   (\overline{u^{t_0}}) 
               \right)^{q}
            } \cdot 
        \nonumber
        \\
       & \cdot     
      \frac{ 
            \left(
               \overline{\omega_0} + 
               \sqrt{-1} \6_{T} \bar\6_{T} 
               (\overline{u^{t_0}}) 
           \right)^{q}
       }
        {
            (\overline{\omega_0})^q
        } - \overline{v}
        \nonumber
        \\
        &= \frac{ 
                   q \left(
                       \overline{\omega_0} + 
                       \sqrt{-1} \6_{T} \bar\6_{T} 
                       (\overline{u^{t_0}}) 
                   \right)^{q-1} 
                \wedge 
                    \sqrt{-1} \6_{T} \bar\6_{T} \overline{v}
           }
            {
                \left(
                   \overline{\omega_0} + 
                   \sqrt{-1} \6_{T} \bar\6_{T} 
                   (\overline{u^{t_0}}) 
               \right)^{q}
            }
        - \overline{v}
        \label{eqn::final_form_of_diff_Cal_before_identific_Laplacian}
    \end{align}

    Recall that given a Hermitian form $\omega$ on the (not necessarily compact) manifold
    $T_i$, the 
    {$(d, \omega)$-Laplacian} 
    is defined as 
    $
        {\Delta_{\omega}^d
            := 
        d\delta_i + \delta_i d}
    $, 
    where $\delta_i$ is the formal adjoint of 
    ${d: \Omega^{*}(T_i) \rightarrow \Omega^{*}(T_i)}$
    with respect to the scalar product induced by $\omega$ on $\Omega^{*}(T_i)$.
    Taking into account the complex structure, we have also the 
    {$(\bar\6_{T}, \omega)$-Laplacian}
    defined as
    ${
        \Delta_{\omega}^{\bar\6_{T}}
            := 
        \bar\6_{T}{\bar\6_{T}}^{*, \omega} 
            + 
        {\bar\6_{T}}^{*, \omega} \bar\6_{T}
    }$,
    where
    $\bar\6_{T}^{*, \omega}$
    is the formal adjoint of $\bar\6_{T}$ with respect to the scalar product
    induced by $\omega$ on $\Omega^{*}(T_i)$.

    Assuming further that $\omega$ is K\"ahler, the 
    {$(d,\omega)$-Laplacian} 
    and the
    \newline
    {$(\bar\6_{T}, \omega)$-Laplacian}
    coincide up to a constant: 
    $
        \Delta_{\omega}^d 
            =
        2 \Delta^{\bar\partial_T}_{\omega}
    $
    (\cite[Theorem~14.6]{moroianu_lectures_kaehler}).

    Coming back to our computation of
    $d^1_{(u^{t_0}, t_0)} \Cal^{-}_T$, denoting by 
    \newline
    ${\Lambda_{\overline{u_i^{t_0}}}: \Omega^*(T_i) \rightarrow \Omega^*(T_i)}$
    the formal adjoint of
    the operator on $\Omega^*(T_i)$ which is the wedge product to the right with
    $
        \overline{\omega_0} 
            + 
        \sqrt{-1} \6_{T} \bar\6_{T}
                \left(
                   \overline{u^{t_0}}
                \right)
    $, 
    we see that \eqref{eqn::final_form_of_diff_Cal_before_identific_Laplacian} reduces to:
    \begin{align*}
        \Lambda_{\overline{u^{t_0}}}
            \left(
                \sqrt{-1} \6_{T} 
                \bar\6_{T}
                (\overline{v})
            \right)
            - 
            \overline{v}
        &= 
        -\bar\6_{T}^{*, 
                    \overline{\omega_0} 
                        + 
                    \sqrt{-1} \6_{T} \bar\6_{T}
                            \left(
                               \overline{u^{t_0}}
                            \right)
                }
        \bar\6_{T} (\overline{v}) - \overline{v}\\
        &= 
        \left(
            -\Delta^{\bar\6_{T}}_
                {
                        \overline{\omega_0} 
                    + 
                        \sqrt{-1} \6_{T} \bar \6_{T} \overline{u^{t_0}}
                }
            - \Id 
            \right)(\overline{v})
    \end{align*}

    But because 
    ${
        \overline{\omega_0} 
            + 
        \sqrt{-1} \6_{T} \bar\6_{T}
            \left(
               \overline{u^{t_0}}
            \right)
    }$
    is by assumption K\"ahler on each $T_i$, we arrive at:
    \[
        d^1_{(u^{t_0}, t_0)} \Cal^{-}_T (\overline{v})
            =
        \left(
            -\Delta^{d}_
                {
                        \overline{\omega_0} 
                    + 
                        \sqrt{-1} \6_{T} \bar \6_{T} \overline{u^{t_0}}
                }
            - \Id 
            \right)(\overline{v})
    \]
    
    The argument so far works for arbitrary $k \in \N$. Suppose now $k > 2 + \codim{(F)}/2$. Then by \cref{thm::transverse_laplacian_plus_identity_invertible}, 
            ${d^1_{(u^{t_0}, t_0)}: 
                \mathcal{C}^{k+2, \alpha}_\bas(M; \mathcal{F}) 
            \rightarrow 
                \mathcal{C}^{k, \alpha}_\bas(M; \mathcal{F})
                }$
    is an isomorphism of Banach spaces. 
    
    Then \cref{thm::inverse_function_theorem_banach_spaces} shows that for $t$ close to $t_0$, there exists a solution $u^t \in B_1$ of the equation
    \[
        \Cal_{T}^{-} (u^t) = tf
    \]
    By the positive definiteness of $\overline{\omega_{0}} + \sqrt{-1}\6_T\bar\6_T \overline{u}$ and \cref{lemma::cal_is_elliptic_wrt_u}, $\Cal^{-}_T$ is elliptic with respect to $u^t$. Noting that $tf$ is in fact not just in $\mathcal{C}^{k,\alpha}_{\bas}(M;\mathcal{F})$ but in $\mathcal{C}^\infty_{\bas}(M;\mathcal{F})$, we can apply \cref{thm::regularity_for_fully_nonlinear} to conclude that $u^t$ is smooth, which completes the proof.
\end{proof}

For the closedness argument in \cref{prop::solution_set_closed}, we will require a uniform estimate in $t$ for the $C^{2, \alpha}_{\bas}$-norm of the solutions of \hyperref[eqn::parametrised_equation]{$(*_{t})$}.

Two intermediate estimates are required to achieve this. 

First, a uniform estimate of the $C^0_{\bas}$-norm of solutions of \hyperref[eqn::parametrised_equation]{$(*_{t})$}. This can be achieved easily: an argument identical to the one in the proof of uniqueness in Step 2 of \cref{thm::transversal_aubin_yau} shows that if $u$ is a solution of \hyperref[eqn::parametrised_equation]{$(*_{t})$}, then at a maximum point $x_{\max}$ of $u$ we have $u(x_{\max})+ tf(x_{\max}) \leq 0$, while at a minimum point $x_{\min}$ of $u$ we have $u(x_{\min})+tf(x_{\min}) \geq 0$. Hence $- \sup (f) \leq \lVert u \rVert_{0} \leq - \inf (f)$, which does not depend on $t$.

The second intermediate estimate is a uniform estimate of the $\mathcal{C}^0$-norm of the Laplacian of solutions for \hyperref[eqn::parametrised_equation]{$(*_{t})$}; there is no necessity for uniform estimates of the gradient. The proof is a verbatim copy of  \cite[Theorem~5.13]{blocki}, which we do not reproduce to avoid overextending the scope of this work; we merely note that the proof involves only local arguments, by working around a maximum point of the solution.

Finally, \cite[Theorem~5.15]{blocki} shows that the $C^{2,\alpha}_{\bas}$-norm of a solution of \hyperref[eqn::parametrised_equation]{$(*_{t})$} can be bound in terms of the $\mathcal{C}^0$-norm of the solution and the $\mathcal{C}^0$-norm of the Laplacian of the solution. The proof is highly involved, but purely local.

Putting everything together, we obtain:

\begin{theorem}
[Yau's A Priori estimate: {\cite{yau_1978_on_ricci_curvature_and_monge_ampere}; {\cite{blocki}}}]
\label{thm::Yau_a_priori_estimate}
    Let $U$ be an open subset of $\C^q$, $\omega \in \Omega^2(U)$ a K\"ahler form on $U$ and $f \in \mathcal{C}^\infty(U)$ a function. Let $\alpha \in (0,1)$. Then there exists a constant $C>0$ depending only on $\lVert f \rVert_{3, \alpha, U}$ with the property that for any $u \in \mathcal{C}^\infty(U)$ such that $\omega + i \6\bar\6 u$ is again a K\"ahler form and, moreover, $\Cal^{-}(u)=tf$ for some $t \in [0,1]$, then the following uniform bounds take place:
    \begin{align}
        \lVert u \rVert_{2, \alpha, U} &\leq C 
            \label{eqn::uniform_Yau_upper_bound} \\
        \omega + \sqrt{-1} \6\bar\6 u &\geq C^{-1} \omega
            \label{eqn::uniform_Yau_lower_bound} 
    \end{align}
\end{theorem}

\begin{proposition}
\label{prop::solution_set_closed}
    Let $(M, \mathcal{F})$ be a transversally K\"ahler foliation.
    Suppose $u^{t_k} \in \mathcal{C}^\infty_\bas(M; \mathcal{F})$ are solutions of \hyperref[eqn::parametrised_equation]{$(*_{t_k})$} and that
    $
        t_k \xrightarrow{k \rightarrow \infty} \bar t \in [0,1]
    $. Then there exists 
    $u^{\bar t} \in \mathcal{C}^\infty_\bas(M; \mathcal{F})$
    solving \hyperref[eqn::parametrised_equation]{$(*_{\bar t})$}.
    \begin{proof}
        Applying \cref{thm::Yau_a_priori_estimate}, \eqref{eqn::uniform_Yau_upper_bound} in each foliated chart $U_i$, we find constants $C_i > 0$ depending on 
        $\lVert \overline{f_i} \rVert_{3, \alpha, T_i}$ 
        and uniformly bounding 
        $\overline{u^{t_k}_i}$. 
        So by definition:
        \[
            \lVert u^{t_k} \rVert_{2, \alpha} \leq \max_i C_i =: C
        \]
        where $C$ depends on $\max_{i} \lVert \overline{f_i} \rVert_{3, \alpha, T_i}$, so on $\lVert f \rVert_{3, \alpha}$.

        Thus 
        $\{ u^{t_k} \}_k$
        is uniformly bounded in 
        $\mathcal{C}^{2, \alpha}_\bas(M; \mathcal{F})$. Let 
        $\alpha' \in (0,1)$ 
        be such that 
        $\alpha' < \alpha$.
        According to \cref{prop::compact_embedding_Holder_spaces}, 
        by passing to a subsequence we can assume that 
        $\{ u^{t_k}\}$
        converges in 
        $\mathcal{C}^{2, \alpha'}(M, \mathcal{F})$,
        say to some 
        $u \in \mathcal{C}^{2, \alpha'}(M ;\mathcal{F})$.

        We show that $u$ is a solution of
        \hyperref[eqn::parametrised_equation]{$(*_{\bar t})$}.
        By \eqref{eqn::uniform_Yau_lower_bound},  
        $\overline{\omega_{0}} + \sqrt{-1}\6_T\bar\6_T \overline{u}$ 
        is positive definite. Moreover, since
        $\lVert u^{t_k} - u \rVert_{2, \alpha'} \xrightarrow{k\rightarrow \infty} 0$
        :
        \[
            \lVert \Cal^{-}_T(u^{t_k}) - \Cal^{-}_T(u) \rVert_{0, \alpha'} \rightarrow 0
        \]
        whence:
        \[
                \lVert \bar{t} f - \Cal^{-}_T(u) \rVert_{L^\infty(M)}
            \leq
                \lVert \bar{t} f - t_k f \rVert_{L^\infty(M)}
                +
                \lVert \Cal^{-}_T(u^{t_k}) - \Cal^{-}_T(u) \rVert_{0, \alpha'} \rightarrow 0
        \]
        Thus $Cal^{-}_B(u) = \bar{t} f$.

        Finally, since $\overline{\omega_{0}} + \sqrt{-1}\6_T\bar\6_T \overline{u}$ 
        is positive definite, \cref{lemma::cal_is_elliptic_wrt_u} guarantees that $\Cal^{-}_T$ is elliptic with respect to $u$. Thus, \cref{thm::regularity_for_fully_nonlinear} ensures $u$ is smooth, which completes the proof.
    \end{proof}
\end{proposition}

\section{The Aubin-Yau theorem for transversally K\"ahler foliations}
\label{sec::aubin_yau_transv_kaehler}
In this section, we state our main theorem and show how to reduce it to \cref{thm::transverse_calabi_yau_equation_has_solutions}. To this end, we proceed in several steps.

\begin{theorem}\label{thm::transversal_aubin_yau}
    Let $(M, \mathcal{F})$ be a homologically orientable, transversally K\"ahler foliation such that $c_1(\nu \mathcal{F}) \in H^2_{\bas}(M;\mathcal{F})$ is negative. Then there exists a unique transversally K\"ahler, transversally Einstein metric with Einstein constant $-1$.
\end{theorem}

\begin{proof}
    \textbf{Step 1. Reformulation of the problem as a differential equation in basic functions.}
        
        As notational convention, whenever $\eta$ is a form on $TM$ which vanishes on $T \mathcal{F}$, understand $\overline{\eta}$ to be the corresponding form on $\nu \mathcal{F}$.
        
        By \cref{def::positive_negative_transversally_kaehler_foliation}, there exists a closed basic $2$-form on $TM$, $\omega_0$, such that $\omega_0$ represents $-c_1(\nu \mathcal{F})$ and for which $\overline{\omega_0}(\cdot, \overline{J} \cdot)$ is positive as a symmetric tensor on $\nu \mathcal{F}$. Then $\omega_0$ is a transversally K\"ahler form on $M$. Let $\rho^{\omega_0}$ be the Ricci form of $\omega_0$ as defined in \cref{prop::transversal_Ricci_curvature_well_def}. By \cref{prop::transversal_Ricci_curvature_def_chern_class}, $\rho^{\omega_0}$ is another form which represents $c_1(\nu \mathcal{F})$, hence $[\omega_0] = [-\rho^{\omega_0}]$.

        Since $\mathcal{F}$ is homologically orientable and transversally K\"ahler, the $\sqrt{-1} \6_T \bar\6_T$-lemma, \cref{lemma::transversal_global_del_del_bar}, yields $f \in \mathcal{C}^\infty_{\bas}(M; \mathcal{F})$ such that
        \begin{equation}
        \label{eqn::initial_ricci_form_and_transv_kaehler_metric_and_function}
            \rho^{\omega_0} = - \omega_0 + i \6_T \bar\6_T f
        \end{equation}

        Now suppose a transversally K\"ahler metric $\omega_0'$ satisfies $\rho^{\omega_0'} = - \omega_0'$. Then in particular $\omega_0'$ also represents $c_1(\nu \mathcal{F})$, so again by \cref{lemma::transversal_global_del_del_bar} there exists $u \in \mathcal{C}^\infty_{\bas}(M;\mathcal{F})$ with 
        \begin{equation}
        \label{eqn::reln_old_transv_kaehler_new_transv_kaehler}
            \omega_0' = \omega_0 + \sqrt{-1}\6_T \bar\6_T u.
        \end{equation} By a standard local argument:
        \begin{equation}            
            \rho^{\omega_0'} - \rho^{\omega_0} = \sqrt{-1} \6_T \bar\6_T 
            \log\frac{
                    \left( 
                        \overline{\omega_0} + \sqrt{-1} \6_T \bar\6_T \overline{u} 
                    \right)^q
                }
                {
                    \overline{\omega_0}^q
                }
        \end{equation}
        Replacing $\rho^{\omega_0'}$ with $-\omega_0'$, then replacing $\omega_0'$ using \eqref{eqn::reln_old_transv_kaehler_new_transv_kaehler} while also using \eqref{eqn::initial_ricci_form_and_transv_kaehler_metric_and_function} to replace $\rho^{\omega}$, we obtain:
        \[
            \6_T \bar\6_T \left( \Cal_T^{-} (u) - f \right) = 0
        \]
        Therefore, the theorem is implied by the claim that if $f \in \mathcal{C}^{\infty}_\bas(M; \mathcal{F})$, then the equation
        \begin{equation}
        \label{eqn::cal_minus_eqn}
            \Cal^{-}_T (u) = f 
        \end{equation}
        has a unique solution $u \in \mathcal{C}^{\infty}_\bas(M; \mathcal{F})$ such that $\omega_0 + \sqrt{-1}\6_T \bar\6_T u$ is positive definite.

        \textbf{Step 2.} Existence of a solution for \eqref{eqn::cal_minus_eqn} is the content of \cref{thm::transverse_calabi_yau_equation_has_solutions}. We are left to prove uniqueness. Suppose $\Cal_T^{-}(u_1) = \Cal_T^{-}(u_2)$ for $u_1, u_2 \in \mathcal{C}^\infty_{\bas}(M; \mathcal{F})$. Denoting $u:= u_2-u_1$, by the properties of $\log$ we have
            \begin{equation}
            \label{eqn::uniqueness_cal}
                \log \frac{
                \left(
                    \overline{\omega_0} + 
                    \sqrt{-1} \6_T\bar\6_T \overline{u_1} +
                    \sqrt{-1} \6_T\bar\6_T \overline{u} 
                \right)^q
                }
                {
                \left(
                    \overline{\omega_0} + \sqrt{-1} \6_T\bar\6_T \overline{u_1} 
                \right)^q
                }
                = u
            \end{equation}
        Since $u_1$ is a solution of \eqref{eqn::cal_minus_eqn}, $\omega_0 + \sqrt{-1} \6_T\bar\6_T u_1$ is a Hermitian form on any smooth local leaf space; denote by $g \in \Symm^2(\nu\mathcal{F})$ its corresponding transversally Hermitian metric. Denote also by $h \in \Symm^2(\nu\mathcal{F})$ the symmetric $2$-tensor associated to $\sqrt{-1}\6_T\bar\6_T u$. Let $x\in M$. Then \eqref{eqn::uniqueness_cal} reads:
           \begin{equation}
            \label{eqn::uniqueness_cal_rewritten_det}
                \log(\det(g_x) + \det(h_x)) - \log(\det(g_x)) = u(x)
            \end{equation}
        By the finite dimensional spectral theorem, we can diagonalise $h_{x}$ in a basis which is orthonormal with respect to $g_{x}$. Then \eqref{eqn::uniqueness_cal_rewritten_det} simplifies to:
        \begin{equation}
        \label{eqn::uniqueness_cal_reduced_to_eigenvals}
            \sum_{\lambda \in \Spec (h_x)} \lambda = u(x)
        \end{equation}

        Note that $\sqrt{-1}\6_T\bar\6_T u = -d(Jdu)$. Consider a local leaf space around $x$, denoted $\overline{U}:= U/\mathcal{F}\rvert_{U}$, and denote by $\nabla$ the Levi-Civita connection of the metric $\overline{g}$ on the local leaf space $\overline{U}$. Denote by $\Alt$ the antisymmetrisation operator $\Alt : \Omega^1(\overline{U}) \otimes \Omega^1(\overline{U}) \rightarrow \Omega^2(\overline{U})$. Note that since $\nabla$ is torsion-free, a calculation involving Cartan's formula for the exterior differential yields that for any $\theta \in \Omega^1(\overline{U})$, $\Alt (\nabla \theta) = d \theta$. Therefore, $d(\overline{J} du) = \Alt ( \nabla (\overline{J} du) )$. Since $\overline{g}$ is K\"ahler on $\overline{U}$, $\nabla \overline{J} = 0$, so $\nabla (du)( \cdot, J \cdot) = \nabla(J du) (\cdot, \cdot)$. Thus,
        \begin{align*}
            \sqrt{-1}\6_T\bar\6_T u (X,Y) 
            &= 
            -d(\overline{J}du) (X,Y)
            =
            -\Alt(\nabla(du)) (X, \overline{J}Y) \\
            &= -\frac{1}{2} \left(
                    \nabla (du)(X, \overline{J}Y) - \nabla(du)(Y, \overline{J}X)
                \right)
        \end{align*}
        In particular, for any $X \in T_{\overline{x}} \overline{U}$ we have:
        \[
            \overline{h}_{\overline{x}}(X,X) 
            = \6_T \sqrt{-1}\6_T\bar\6_T u (X,\overline{J}X)
            = \frac{1}{2} \left(
                    \nabla (du)(X, X) + \nabla(du)(\overline{J}X, \overline{J}X)
                \right)
        \]

        Take now $x = x_{\max}$ to be a maximum point of $u$. Then 
        \[
            \overline{h}_{\overline{x}}(X,X)
            = \frac{1}{2} \left(
                    \Hess^{\overline{u}} (X, X) + \Hess^{\overline{u}}(\overline{J}X, \overline{J}X)
                \right)
                \leq 0
        \]
        Combining this with \eqref{eqn::uniqueness_cal_reduced_to_eigenvals}, it follows that $\max_M u \leq 0$. Analogously, we prove that $\min_M u \geq 0$, which shows that $u_1 = u_2$, concluding the proof.
\end{proof}

\section{Application - Vaisman manifolds}
\label{sec::application_vaisman}
\subsection{Locally Conformally K\"ahler and Vaisman Manifolds}
We recall here the basic properties of locally conformally K\"ahler (LCK for short) and Vaisman manifolds.

\begin{definition}\label{def::lck}
A complex manifold $M$ with real Hermitian form $\omega$ is called \textit{locally conformally K\"ahler} if there exists a closed $1$-form $\theta$ on $M$, called the Lee form, such that:
\begin{equation} \label{eq::lck_domega}
    d\omega = \omega \wedge \theta
\end{equation}
\end{definition}

Among LCK manifolds, a certain class yields itself to techniques that are not generally available to study LCK manifolds, among which we will leverage foliation theory applied to the canonical foliation.

\begin{definition}\label{def::vaisman}
An LCK manifold $(M, \omega, \theta)$ is called \textit{Vaisman} if its associated Lee form is parallel with respect to the Levi-Civita connection of its Hermitian metric:
\[
\nabla \theta = 0
\]
\end{definition}

\begin{proposition}[{\cite{vaisman_82_generalized_Hopf}}]
    Let $(M, \omega, \theta)$ be a Vaisman manifold. Then the Lee field, $\theta^\sharp$, is holomorphic. In particular, denoting the complex structure on $M$ by $J$, the distribution generated by $\theta^\sharp$ and $J \theta^\sharp$ is integrable. 
\end{proposition}

\begin{definition}\label{def::canonical_foliation_vaisman}
    Let $(M, \omega, \theta)$ be a Vaisman manifold. The foliation associated to the distribution generated by $\theta^\sharp$ and $J \theta^\sharp$ is called \textit{the canonical foliation} and will be denoted by $\Sigma$.
\end{definition}

\begin{remark}
    If $M$ is compact, the definite article in \cref{def::canonical_foliation_vaisman} is all the more justified: by \cite{tsukada_94_holomorphic_forms}, the direction of the Lee field depends only on the complex structure of a manifold of Vaisman type. In other words, if $(M, J)$ is a complex manifold and $(\omega_1, \theta_1)$, $(\omega_2, \theta_2)$ are two Vaisman structures on $M$, then $(\theta_1)^\sharp = f (\theta_2)^\sharp$ with $f \in \mathcal{C}^\infty (M)$. Therefore, the canonical foliation also depends only on the complex structure of a compact manifold of Vaisman type.
\end{remark}

\begin{theorem}
[
\cite{vaisman_82_generalized_Hopf},
\cite{verbitsky_vanishing_theorems}
]
\label{thm::characterisation_omega_0}
    Let $(M, \omega, \theta)$ be a Vaisman manifold and set
    $
        {\omega_0 := - dJ \theta.}
    $
    Then 
    \begin{equation}
    \label{eqn::formula_Vaisman_metric_wrt_Lee_form}
        {\omega = \omega_0 + \theta \wedge J\theta}
    \end{equation}
    and $\omega_0$ is a semi-positive, closed $(1,1)$-form, such that
    $
        {\Sigma = \ker \omega_0.}
    $
\end{theorem}

\begin{remark}
    \cref{thm::characterisation_omega_0} shows that the canonical foliation of a Vaisman manifold is a transversally K\"ahler foliation (\cref{def::transversally_kaehler_foliation}).
\end{remark}

Another way to regard LCK manifolds involves seeing the Hermitian form as valued in a rank-$1$ local system with flat connection defined by setting the Lee form as connection form, and asking of the Hermitian form seen as such to be closed with respect to the differential operator induced by the connection. The $(0,1)$ part of the connection gives a holomorphic structure on the local system. One of the advantages of this point of view is that the curvature of the Chern connection of the aforementioned holomorphic structure is known and has a simple formula: it is a multiple of $\omega_0$.

\begin{definition}\label{def::local_system}
    A vector bundle $E$ equipped with a flat connection $\nabla$ is called a \textit{local system}.
\end{definition}

Starting with an LCK manifold $(M, \omega, \theta)$, choose any oriented real line bundle $L$ with a global non-degenerate section $\psi$. Set $\nabla(f \psi):= df \otimes \psi - \theta \otimes \psi$. Then $\omega \otimes \psi \in \Omega^2(M, L)$ satisfies ${d_\nabla (\omega \otimes \psi) = 0}$ and $\nabla$ is flat, since $\theta$ is closed. It can also be shown that starting with a $d_\nabla$-closed $L$-valued $2$-form in a rank-$1$ local system $(L, \nabla)$, one arrives at the classical definition of LCK manifolds by taking the Lee form to be the connection form of $\nabla$ (see \cite[Chapter~3]{lck_book}).

\begin{theorem}{\cite{verbitsky_vanishing_theorems}}
\label{thm::curvature_of_weight_bundle}
    Let $(M, \omega, \theta)$ be an LCK manifold defined by the local system $(L, \nabla)$. Let ${}^C\nabla$ be the Chern connection of the holomorphic structure given by $\nabla^{0,1}$. Then the trace of the curvature of ${}^C\nabla$ is 
    $\Tr \left( R^{{}^C\nabla} \right) = 2 \sqrt{-1} dJ \theta$.
\end{theorem}
\subsection{An Aubin-Yau theorem for Vaisman manifolds}
Let $M$ be a Vaisman manifold. In this section we show how to apply \cref{thm::transversal_aubin_yau} for the canonical foliation of $M$ to obtain a new Vaisman metric, provided that we have enough control over the associated weight bundle. Without this control, we can obtain a transversally K\"ahler metric which is transversally Einstein, but it may fail to come from a Vaisman one (see also \cref{rmk::general_vaisman_transversally_Kaehler}).

First, we recall:

\begin{definition}
\label{def::chern-ricci-form}
     For any Hermitian manifold $(M, \omega)$, the \textit{Chern-Ricci form} of $\omega$ is the Chern curvature form of the Chern connection with respect to the standard complex structure and the induced Hermitian metric on the determinant bundle $K_M^*$ of $M$. We denote the Chern-Ricci form of $\omega$ by $\rho^\omega$. Since $\rho^{\omega}$ is closed, it defines a cohomology class $[ \rho^{\omega} ] \in H^2_{\deRham}(M, \R)$, which satisfies $\frac{1}{2\pi} [\rho^{\omega}] = c_1(K^*_M)$.
\end{definition}

\begin{remark}
\label{rmk::chern_ricci_form_is_transverse_Ricci_on_Vaisman}
    Suppose $(M, \omega, \theta)$ is a Vaisman manifold with transversally K\"ahler form $\omega_0 = -d J \theta$ (\cref{thm::characterisation_omega_0}).
    Denote as in \cref{def::canonical_foliation_vaisman} the canonical foliation by $\Sigma$. Then $\nabla^{LC}\rvert_{\Sigma}$ is flat, because $\Sigma$ is generated by Killing vector fields. Therefore, the transverse Ricci form $\rho^{\omega_0}$ (\cref{def::transversal_Ricci_form}) coincides with the Chern-Ricci form $\rho^{\omega}$ (\cref{def::chern-ricci-form}). 
    In particular, $[ \rho^{\omega} ] = [ \rho^{\omega_0} ] \in H^2_{\bas}(M; \Sigma)$.
\end{remark}

The following result has been proved in \cite[Theorem~6.3]{istrati_2025_Vaisman_manifolds_vanishing_first_chern_class}. The hypotheses in loc. cit. are that $c_1(M) = 0$ and $-c_1(M) \in H^{1,1}_{BC}(M,\R)$ can be represented by a transversally K\"ahler form. By the exact sequence \cite[(9)]{istrati_2025_Vaisman_manifolds_vanishing_first_chern_class}, the kernel of the natural map from the $(1,1)$-Bott-Chern cohomology group, $H^{1,1}_{BC}(M,\R)$, to $H^2_{\deRham}(M,\R)$, is $1$-dimensional and generated by $[-dJ\theta]$.
This readily shows that this formulation of the hypotheses is equivalent with the one we give below, in \cref{thm::Vaisman_Aubin_Yau}.
The method employed in \cite{istrati_2025_Vaisman_manifolds_vanishing_first_chern_class} consists of first proving a Weitzenb\"ock-type formula for a Vaisman manifold $M$, then using it to show that in the already-discussed hypotheses $M$ must be quasi-regular. We present a simplified, alternative proof, using our main result \cref{thm::transversal_aubin_yau}. 

\begin{theorem}\label{thm::Vaisman_Aubin_Yau}
    Let $(M, \omega, \theta)$ be a Vaisman manifold and consider $L$ the weight bundle of $M$. Suppose that $c_1(L) = c_1(K_M^{\otimes 2}) \in H^2_{\bas}(M;\Sigma)$. 
    Then there exists a unique Vaisman metric $\omega'$ with Lee form $\theta'$ such that $[\theta'] = [\theta] \in H^1_{\deRham}(M,\R)$ and for which the Chern-Ricci curvature $\rho^{\omega'}$ satisfies $\rho^{\omega'} = - \omega_0'$.
    \begin{proof}
        The hypothesis $c_1(L) = c_1(K_M^{\otimes 2})$ translates by \cref{rmk::chern_ricci_form_is_transverse_Ricci_on_Vaisman} and by  \cref{thm::curvature_of_weight_bundle} to the fact that $[\rho^{\omega_0}] = [-\omega_0]$.
        
        We now show the uniqueness part of \cref{thm::Vaisman_Aubin_Yau}. Suppose $\omega'$ and $\omega''$ are Vaisman metrics satisfying $\rho^{\omega'} = - \omega'_0$ and $\rho^{\omega''} = - \omega''_0$ and moreover $[\theta'] = [\theta''] = [\theta]$. Then $\rho^{\omega'_0} = - \omega'_0$ and $\rho^{\omega''_0} = - \omega''_0$. Taking into account also the fact that $[\rho^{\omega'_0}] = [\rho^{\omega''_0}] \in H^2_{\bas}(M;\Sigma)$, we obtain by the uniqueness part in \cref{thm::transversal_aubin_yau} that $\omega'_0 = \omega''_0$. But since $[\theta'] = [\theta'']$, for some $u \in \mathcal{C}^\infty(M)$ we have $\theta'' = \theta' + du$. Combining these two datums we obtain that $dd^c u = 0$, whence from the maximum principle $u$ must be constant.
        Therefore, up to a constant, $\omega' = \omega''$.

        Now we show existence. By \cref{prop::transversal_Ricci_curvature_def_chern_class}, our  hypothesis ensures that $c_1(\nu\Sigma) \in H^2_{\bas}(M;\Sigma)$ is transversally negative definite. Then by the existence part in \cref{thm::transversal_aubin_yau}, we find $\omega_0'$ transversally K\"ahler such that $\rho^{\omega_0'} = -\omega_0'$. But then we also have $[-\omega_0'] = c_1(\nu \Sigma) = [dJ\theta]$, whence for some $u \in \mathcal{C}^\infty(M)$,
        \begin{equation}
            \label{eqn::reln_new_transv_kaehler_form_and_old}
            \omega_0' = -dJ\theta - dJdu = -dJ (\theta + du)
        \end{equation}
        Set $\theta' := \theta + du$. Then by \eqref{eqn::reln_new_transv_kaehler_form_and_old} the Hermitian form $\omega':= \omega_0' + \theta' \wedge J \theta'$ is LCK with Lee form $\theta'$ satisfying $[\theta'] = [\theta]$.

        We are left to show that $\omega'$ is Vaisman. We show first that $(\theta')^\sharp$ is holomorphic. For this, we show that it is in the span of the two holomorphic vector fields $\theta^\sharp$ and $J \theta^\sharp$. By the Cartan formula, the group of automorphsims given by the flows of $\theta^\sharp$ and $J\theta^\sharp$ acts trivially on $H^1_{\deRham}(M)$, so the pullback of $\theta'$ by any automorphism in this group differs from $\theta'$ by an exact form, which must vanish since $\omega'_0$ is basic. Therefore $(\theta')^\sharp$ is holomorphic.
        By \cite[Proposition~1]{istrati_2019_existence_criteria_for_special_LCK}, the constructed metric is thus Vaisman.
    \end{proof}
\end{theorem}

\begin{corollary}
\label{rmk::vaisman_transversally_kaehler_alpha_explanation}
    Let $(M, \omega, \theta)$ be a Vaisman manifold and suppose that for some $\alpha \in \R_{>0}$, $\alpha c_1(L) = c_1(K_M^{\otimes 2}) \in H^2_{\bas}(M;\Sigma)$. 
    Then there exists a unique Vaisman metric $\omega'$ with Lee form $\theta'$ such that $[\theta'] = [\alpha \theta] \in H^1_{\deRham}(M,\R)$ and for which the Chern-Ricci curvature $\rho^{\omega'}$ satisfies $\rho^{\omega'} = - \alpha \omega_0'$.
\begin{proof}
     Consider $\widetilde{\omega}$ to be the Vaisman metric obtained via \eqref{eqn::formula_Vaisman_metric_wrt_Lee_form} for the Lee form 
     $\widetilde{\theta} := \alpha \theta$. Then $[dJ \widetilde{\theta}] = c_1(\nu \Sigma) \in H^2_{\bas}(M;\Sigma)$, because the transversally K\"ahler metric of $\widetilde{\omega}$ differs from that of $\omega$ by multiplication with a constant, which does not change the transverse Ricci curvature.
     Thus we can apply \cref{thm::Vaisman_Aubin_Yau} for $(\widetilde{\omega}, \widetilde{\theta})$ to obtain a unique Vaisman metric $\omega''$ with Lee form $\theta''$ such that $[\theta''] = [\alpha \theta]$ and $\omega_0'' = - \rho^{\omega''}$. Then $\omega':= \alpha \omega''$ has Lee form $\theta''$ and satisfies $\omega' = - \alpha \rho^{\omega'}$, again by invariance of the transverse Ricci form with respect to rescaling of the metric.
\end{proof}
\end{corollary}

\begin{remark}
\label{rmk::general_vaisman_transversally_Kaehler}
    The proof of \cref{thm::Vaisman_Aubin_Yau} shows in fact that starting with any Vaisman manifold such that $c_1(\nu \Sigma)$ is transversally negative definite, we find a unique basic $2$-form $\omega'_0$ which is transversally K\"ahler-Einstein. However, in general we have no way to control whether $\omega'_0$ is the transversally K\"ahler metric associated to a Vaisman metric or not.
\end{remark}

\hfill

{\small
	
	\noindent {\sc Vlad Marchidanu\\
		University of Bucharest, Faculty of Mathematics and Informatics, \\
            14 Academiei str., 70109 Bucharest, Romania\\
            also: \\
            Institute of Mathematics “Simion Stoilow” of the Romanian Academy \\
            21, Calea Grivitei Str. 010702-Bucharest, Romania \\
		\tt marchidanuvlad@gmail.com}
}
\end{document}